\documentclass[12pt,twoside]{article}
\usepackage{mathrsfs}
\usepackage{amsfonts}
\usepackage{amsmath}
\usepackage{amssymb}
\usepackage{amsthm,amscd}
\usepackage{tikz}
\usetikzlibrary{cd}
\usepackage{biblatex}
\bibliography{ref}

\newtheorem{theorem}{Theorem}[section]
\newtheorem{conjecture}{Conjecture}[section]
\newtheorem{example}{Example}[section]
\newtheorem{lemma}{Lemma}[section]
\newtheorem{definition}{Definition}[section]
\newtheorem{corollary}{Corollary}[section]
\newtheorem{proposition}{Proposition}[section]
\newtheorem{remark}{Remark}[section]

\newcommand{\C}{\mathbb{C}}
\newcommand{\R}{\mathbb{R}}
\newcommand{\Q}{\mathbb{Q}}
\newcommand{\Z}{\mathbb{Z}}
\renewcommand{\O}{\mathcal{O}}
\renewcommand{\P}{\mathbb{P}}
\renewcommand{\H}{\mathcal{H}}
\newcommand{\Amp}{\mathrm{Amp}}
\newcommand{\Kah}{\mathrm{Kah}}
\newcommand{\A}{\mathcal{A}}
\newcommand{\vol}{\mathrm{Vol}}

\title{Lefschetz theorems, Hodge-Riemann relations and Ample vector bundles}
\author{Yiran Lin\footnote{QiuZhen College, Tsinghua University}}
\date{}

\begin{document}
\maketitle

\begin{abstract}
    We introduce a new Hermitian metric on the cohomology ring of compact K\"ahlerian manifolds with a pair $(v,w)$ satisfying certain Hodge-Riemann relations. An Hermitian metric on the exterior algebra of the cotangent bundle is also defined and we establish the corresponding theory of harmonic forms, relating the global metric and local metric. This generalizes the classical Hodge theory. As an immediate application we give a new proof of Dinh-Nguyen's theorem on the Hodge-Riemann relations for mixed K\"ahler classes. We give several other applications to the Lefschetz property and Hodge-Riemann relations of Chern classes of ample vector bundles.
\end{abstract}

\tableofcontents

\section{Introduction}
Let $X$ be a compact K\"ahlerian manifold of dimension $n$, $w\in H^{1,1}(X,\R)$ a K\"ahler class on $X$. Let $p,q,k$ be nonnegative integers satisfying $p+q+k=n$. The classical hard Lefschetz theorem states that
\[-\wedge w^k: H^{p,q}(X)\rightarrow H^{n-q,n-p}(X)\]
is an isomorphism. The classical Hodge-Riemann bilinear relation states that the Hermitian bilinear form on $H^{p,q}(X)$
\[\left<a,b\right>=(-1)^q\sqrt{-1}^{(p+q)^2}\int_Xa\bar{b}w^k\]
is positive definite on the subspace
\[P^{p,q}(X)=\ker\left(-\wedge w^{k+1}: H^{p,q}(X)\rightarrow H^{n-q+1,n-p+1}(X)\right).\]
Motivated by the theory of mixed volumes in convex geometry, attempts of generalizations of the classical hard Lefschetz theorem and Hodge Riemann relation to the case of mixed K\"ahler classes were made by Gromov \cite{G}, Timorin \cite{T}, etc. A satisfactory answer was given by T.-C. Dinh and V.-A. Nguyen \cite{DN2}. The main theorem of Dinh-Nguyen states that for K\"ahler classes $w_1,\cdots,w_k,w\in H^{1,1}(X,\R)$, we have
\begin{itemize}
    \item[(1)](Hard Lefschetz)
    \[-\wedge w_1\cdots w_k: H^{p,q}(X)\rightarrow H^{n-q,n-p}(X)\]
    is an isomorphism.
    \item[(2)](Hodge-Riemann relation) The Hermitian bilinear form on $H^{p,q}(X)$
    \[\left<a,b\right>=(-1)^q\sqrt{-1}^{(p+q)^2}\int_Xa\bar{b}w_1\cdots w_k\]
    is positive definite on the subspace
    \[P^{p,q}(X)=\ker\left(-\wedge w_1\cdots w_kw: H^{p,q}(X)\rightarrow H^{n-q+1,n-p+1}(X)\right).\]
\end{itemize}
Motivated by the work of S. Bloch and D. Gieseker on the positivity of the Chern classes of an ample vector bundle \cite{BG}, there were attempts of generalizations of the above theorems to Chern classes of ($\R$-twisted) ample vector bundles. See, for example, \cite{RT}\cite{LZ}.

The goal of this paper is to prove the hard Lefschetz theorem and Hodge Riemann relation for several $\R$-twisted ample vector bundles and for mixed K\"ahler classes together with an ample vector bundle. We introduce some new techniques and overcome several technical difficulties. In particular we obtain new (in some sense, more natural) proofs of the theorem of Dinh-Nguyen and the theorem of Timorin. First of all we have the following
\begin{theorem}\label{theorem 1}
    Let $X$ be a smooth projective variety of dimension $n$, $E_1,\cdots,E_k$ ample $\R$-twisted vector bundles of rank $e_1,\cdots,e_k$ on $X$. Let $p,q$ be nonnegative integers satisfying $p+q+e_1+\cdots+e_k=n$. Then
        \[-\wedge c_{e_1}(E_1)\cdots c_{e_k}(E_k):H^{p,q}(X)\rightarrow H^{n-q,n-p}(X)\]
        is an isomorphism.
\end{theorem}
The full statement of theorem \ref{theorem 1} also includes the Hodge-Riemann relation (see theorem \ref{theorem 1 HR} below). We refer to \cite{L} section 6.2, 8.1.A and \cite{RT} section 2.4 for basics on $\R$-twisted vector bundles. Theorem \ref{theorem 1} is trivial when the vector bundles have no $\R$-twist.

For Combinations of ample vector bundles and mixed K\"ahler classes, we have the following
\begin{theorem}\label{theorem 3}
    Let $X$ be a compact K\"ahlerian manifold of dimension $n$, $E$ an ample vector bundle on $X$ of rank $e$, $w_1,\cdots,w_k\in H^{1,1}(X,\R)$ K\"ahler classes. Let $p,q$ be nonnegative integers satisfying $p+q+e+k=n$. Then
        \[-\wedge c_e(E)w_1\cdots w_k:H^{p,q}(X)\rightarrow H^{n-q,n-p}(X)\]
        is an isomorphism.
\end{theorem}
\begin{remark}
    In theorem \ref{theorem 3} we allow the ample vector bundle $E$ to be $\R$-twisted, or more generally `K\"ahler' vector bundles (see section 8). These are natural generalizations of K\"ahler classes. A K\"ahler line bundle is nothing but a K\"ahler class.
\end{remark}
We also have the following generalization of the theorem of Dinh-Nguyen to the case when not all the classes are strictly positive.
\begin{theorem}\label{theorem 2}
    Let $X,Y$ be compact K\"ahlerian manifolds of dimension $m,n$ respectively and $\pi:Y\rightarrow X$ a holomorphic map such that $\pi^*:H^*(X,\C)\rightarrow H^*(Y,\C)$ makes $H^*(Y,\C)$ into a free $H^*(X,\C)$-module. Let $s\leq r$ be non-negative integers, $w_1,\cdots,w_s\in\Kah(X),\ w_{s+1},\cdots,w_r\in\Kah(Y)$. Let $v=\pi^*w_1\cdots\pi^*w_s$, $u=v\cdot w_{s+1}\cdots w_r$. Let $p,q$ be non-negative integers satisfying $p+q+r=n$. Then
    \[\ker\left(-\wedge u:H^{p,q}(Y)\rightarrow H^{p+r,q+r}(Y)\right)\subset\ker\left(-\wedge v: H^{p,q}(Y)\rightarrow H^{p+s,q+s}(Y)\right).\]
\end{theorem}
The following is another generalization of the theorem of Dinh-Nguyen.
\begin{theorem}\label{theorem 4}
    Let $X$ be a compact K\"ahlerian manifold of dimension $n$. Let $w_1,\cdots,w_k\in H^{1,1}(X,\R)$ be K\"ahler classes. Let $p,q,r$ be nonnegative integers satisfying $p+q+r+k=n$. Let $c\in H^{r,r}(X,\R)$. Suppose that there is a compact K\"ahlerian manifold $Y$ of dimension $n+r-1$ and a holomorphic map $\pi:Y\rightarrow X$ such that $\pi^*c$ factors as $\pi^*c=w\cdot b$, with $w\in H^{1,1}(Y,\R)$ being a K\"ahler class and $b\in H^{r-1,r-1}(Y,\R)$ satisfying the property that $\pi^*(-)\cdot b: H^{p,q}(X)\rightarrow H^{p+r-1,q+r-1}(Y)$ is injective. Then
    \[-\wedge c\cdot w_1\cdots w_k: H^{p,q}(X)\rightarrow H^{n-q,n-p}(X)\]
    is an isomorphism.
\end{theorem}
Theorem \ref{theorem 1} is relatively simple and is proven in section 3, but it already contains some new ideas of this paper. Theorem \ref{theorem 3}-\ref{theorem 4} take up most of this paper. The proof of theorem \ref{theorem 2} and theorem \ref{theorem 4} are basically the same, and theorem \ref{theorem 3} will be derived easily as a corollary of either of them. In fact, theorem \ref{theorem 3} is the motivation of both theorem \ref{theorem 2} and \ref{theorem 4}. It is worth mentioning that \cite{X} contains theorems of similar kind as theorem \ref{theorem 2}. But the situation they consider is much simpler than ours. Basically they add the additional assumption that $\mathrm{codim}(Y/X)$ is small. In that case the pointwise Hodge-Riemann relation holds, so their theorem is a simple application of the local-to-global argument of \cite{DN3} after establishing the local result. But in our case the local Hodge-Riemann relation does not hold. We need to improve both the local-to-global part and the local part. Overcoming these difficulties is necessary for our application to theorem \ref{theorem 3}.

\subsection{Outline I}
The statement of the main theorems above does not involve much about metrics, but we start with a discussion about metrics in section 2. In fact, we have the following principle (or rather called a philosophy) underlying all the proofs of this paper:

\textbf{Principle:} One should always try to prove a quantitative version of a statement.

Usually this principle appears in the following form: if there is a proposition stating that a homomorphism between two vector spaces is an isomorphism, then a much more useful version of the proposition would be that the homomorphism is a quasi-isometry with respect to some metrics on the two vector spaces. Here we require that the quasi-isometry constant to be independent of some choices (otherwise every isomorphism between two finite dimensional real vector spaces would be a quasi-isometry with respect to any choice of metrics). In other words the operator norm of both the homomorphism and its inverse are bounded. We shall refer to this principle as the `quantitative principle'. This is well illustrated by the proof of theorem \ref{theorem 1}, which we now sketch.

We first state the full version of theorem \ref{theorem 1}, with the Hodge-Riemann relation included.
\begin{theorem}\label{theorem 1 HR}
    Let $X$ be a smooth projective variety of dimension $n$, $E_1,\cdots,E_k$ ample $\R$-twisted vector bundles of rank $e_1,\cdots,e_k$ on $X$. Let $p,q$ be nonnegative integers satisfying $p+q+e_1+\cdots+e_k=n$. Then
    \begin{itemize}
        \item[(1)] (Hard Lefschetz)
        \[-\wedge c_{e_1}(E_1)\cdots c_{e_k}(E_k):H^{p,q}(X)\rightarrow H^{n-q,n-p}(X)\]
        is an isomorphism.
        \item[(2)] (Hodge-Riemann relation) Let $h\in N^1(X)_\R$ be an arbitrary ample class. Then The Hermitian form on $H^{p,q}(X)$
        \[\left<a,b\right>=(-1)^q\sqrt{-1}^{(p+q)^2}\int_Xa\bar{b}\cdot c_{e_1}(E_1)\cdots c_{e_k}(E_k)\]
        is positive definite on the subspace
        \[P^{p,q}(X)=\ker\left(-\wedge c_{e_1}(E_1)\cdots c_{e_k}(E_k)h:H^{p,q}(X)\rightarrow H^{n-q+1,n-p+1}(X)\right).\]
    \end{itemize}
\end{theorem}
It is known how theorem \ref{theorem 1 HR} follows from theorem \ref{theorem 1} (by a standard continuity argument). See, for example, \cite{RT}\cite{LZ}. For completeness we will give a proof of theorem \ref{theorem 1 HR}.

Theorem \ref{theorem 1} is trivial when $E_1,\cdots,E_k$ are all vector bundles (without $\R$-twist). Indeed, let $E=E_1\oplus\cdots\oplus E_k,\ e=e_1+\cdots e_k$, then $c_e(E)=c_{e_1}(E_1)\cdots c_{e_k}(E_k)$. The theorem is reduced to the Bloch-Giesecker theorem. If $E_1,\cdots,E_k$ are all $\Q$-twisted vector bundles, then it is well-known that there exists a finite dominant morphism $\pi:Y\rightarrow X$, with $Y$ smooth, such that $\pi^*E_1,\cdots,\pi^*E_k$ are all vector bundles. (Usually $\pi$ is called a Bloch-Geiseker covering.) It is not difficult to prove that $\pi^*:H^*(X,\C)\rightarrow H^*(Y,\C)$ is injective. This reduces the problem to the previous case. However, the theorem is substantially more difficult when we allow the vector bundles to be $\R$-twisted. Although $N^1(X)_\Q$ is dense in $N^1(X)_\R$, we may approximate the $\R$-twisted vector bundles $E_1,\cdots,E_k$ by $\Q$-twisted vector bundles, but the homomorphism $-\wedge c_{e_1}(E_1)\cdots c_{e_k}(E_k)$ may degenerate when passing to limit. So the statement on $\Q$-twisted bundles does not imply the statement on $\R$-twisted bundles directly. To overcome this difficulty, we evoke the quantitative principle, and prove the following:
\begin{proposition}\label{proposition 1}
    Let $X$ be a smooth projective variety of dimension $n$ and let $\omega_0$ be a K\"ahler metric on $X$. Let $M>0$ be a positive real number. Then there is a constant $c=c(X,\omega_0,M)>0$ with the following property. Let $E_1,\cdots,E_k$ be ample $\R$-twisted vector bundles on $X$ of rank $e_1,\cdots,e_k$ respectively, satisfying the condition $\|c_i(E_j)\|_{\omega_0}\leq M,\ \forall i,j$. Denote by $c_e=c_{e_1}(E_1\left<\omega_0\right>)\cdots c_{e_k}(E_k\left<\omega_0\right>)$. Let $p,q$ be nonnegative integers satisfying $p+q+e_1+\cdots+e_k=n$. Then $\forall a\in H^{p,q}(X)$ we have
    \begin{align}\label{eq31}
        \|c_e\cdot a\|_{\omega_0}\geq c\|a\|_{\omega_0}.
    \end{align}
\end{proposition}
Equivalently, (\ref{eq31}) says that $-\wedge c_e:H^{p,q}(X)\rightarrow H^{n-q,n-p}(X)$ is an isomorphism and the inverse has norm bounded from above by $c^{-1}$.

It is not hard to see that theorem \ref{theorem 1} follows from this proposition. The advantage of the proposition is that what we need to prove is an inequality which is closed under taking limit: if $E_1^{(m)},\cdots,E_k^{(m)}$ is a sequence of $\Q$-twisted vector bundles (of rank $e_1,\cdots,e_k$) such that $c_i(E_j^{(m)})$ converges to $c_i(E_j)$ for each $i,j$, and the inequality (\ref{eq31}) holds with $E_1\cdots,E_k$ replaced by $E_1^{(m)},\cdots,E_k^{(m)}$ for each $m$, then the inequality also holds for $E_1\cdots,E_k$. Since $N^1(X)_\Q$ is dense in $N^1(X)_\R$, to prove the proposition, it suffices to prove it assuming that $E_1,\cdots,E_k$ are all $\Q$-twisted vector bundles.

So now we can proceed as before and choose a Bloch-Giesecker covering $\pi:Y\rightarrow X$ such that $\pi^*E_1,\cdots,\pi^*E_k$ are vector bundles. Let $E=\pi^*E_1\oplus\cdots\oplus\pi^*E_k$ and we try to apply the Bloch-Giesecker theorem to $E\left<\omega_0\right>$. The Bloch-Giesecker theorem is `not quantitative', stating only that $-\wedge c_e(E\left<\omega_0\right>):H^{p,q}(Y)\rightarrow H^{n-q,n-p}(Y)$ is an isomorphism. In order to prove proposition \ref{proposition 1}, we unwind the proof of the Bloch-Giesecker theorem and make it `quantitative'. The Bloch-Giesecker theorem is proved by applying the hard Lefschetz theorem to $\O_{\P(E\left<\omega_0\right>)}(1)$ on the projective bundle $\P(E)$ of $E$ over $Y$. The corresponding quantitative version holds: $-\wedge c_1(\O_{\P(E\left<\omega_0\right>)}(1)):H^{p+e-1,q+e-1}(\P(E),\C)\rightarrow H^{p+e,q+e}(\P(E),\C)$ is a quasi-isometry (see equation (\ref{isometry0}) in section 2). This is the main idea. The detailed estimates will be given in section 3.

\subsection{Outline II}

Sections 4-8 are devoted to the proof of theorem \ref{theorem 3}-\ref{theorem 4}. Theorem \ref{theorem 3} will be derived as a corollary of theorem \ref{theorem 2}, where $\pi:Y\rightarrow X$ will be taken to be $\pi:\P(E)\rightarrow X$. Since to prove theorem \ref{theorem 3} one would need to deal with the nef classes $\pi^*w_i$, it is natural to ask for generalization of the mixed hard Lefschetz theorem to nef classes. This motivates theorem \ref{theorem 2}. The best one can hope is the following statement: if $h_1,\cdots,h_r$ are nef or semi-ample classes on $X$, and $h_i$ is ample for $i>s$, then $\ker(-\wedge h_1\cdots h_r)\subset\ker(-\wedge h_1\cdots h_s)$, in appropriate dimension. We do not hope this statement to be true in such generality, but succeed in proving it under the additional assumption that $\pi^*: H^*(X)\rightarrow H^*(Y)$ makes $H^*(Y)$ into a free $H^*(X)$-module. We will point out where this assumption is used in the sketch of proof of theorem \ref{theorem 2} below.

The proof of theorem \ref{theorem 2} is quite difficult and it splits into two parts: the local part and the local-to-global part. In this subsection we outline these two parts and then sketch the proof of theorem \ref{theorem 2}. Before that we give some key observations. These observations motivate the definition of the metric $\|\cdot\|_{(v,w)}$ below.

\textbf{Observation 1:} The metric $\|\cdot\|_{\omega}$ on $H^{p,q}(X)$ is completely determined by the cohomology ring $H^{*,*}(X)$. In particular it depends only on the K\"ahler class $[\omega]$ (independent of the particular choice of the K\"ahler form $\omega$ in this class).

As equations (\ref{compute the metric by Q1}) and (\ref{compute the metric by Q2}) in section 2 show, in order to compute $\|\cdot\|_{\omega}$ on $H^{p,q}(X)$, one only need to know the cup product structure on $H^{*,*}(X)$ and the pairing with the fundamental class $[X]$, $\int_X:H^{n,n}(X)\rightarrow\C$ (the details about the variety $X$ and the K\"ahler form $\omega$ can be forgotten). Therefore, if $w\in\Kah(X)$ is any K\"ahler class, we can write $\|\cdot\|_w$ for $\|\cdot\|_\omega$ where $\omega$ is any K\"ahler form in the K\"ahler class $w$.

\textbf{Observation 2:} Let $h\in\Amp(X)$ and $Y$ a smooth divisor representing $h$. Then for $p+q<\dim X-1$, the restriction map $H^{p,q}(X)\rightarrow H^{p,q}(Y)$ is (almost) an isometry, where we use the metrics induced by $h$ and $h|_Y$ respectively on the two spaces.

This is a quatitative version of the Lefschetz hyperplane theorem, which says that the restriction map above is an isomorphism. This observation will follow from lemma \ref{lemma: restriction quasi-isometry 1}.

\textbf{Hyperplane trick}

The following quick proof of the mixed Lefschetz theorem for ample classes surves as a motivation.

The mixed Lefschetz theorem for ample classes: Let $h_1,\cdots,h_k\in N^1(X)$ be ample classes, then $-\wedge h_1\cdots h_k:H^{p,q}(X)\rightarrow H^{n-q,n-p}(X)$ is an isomorphism, where $p+q+k=n$.
\begin{proof}
    Replacing $h_1,\cdots,h_k$ by a multiple we may assume that they are very ample. Take transversally intersecting smooth divisors $Y_1,\cdots,Y_{k-1}$ representing $h_1,\cdots h_{k-1}$. Let $Y=Y_1\cap\cdots\cap Y_k$. Then it is easy to see $\int_Xa\bar{b}\cdot h_1\cdots h_k=\int_Ya\bar{b}\cdot h_k$ for $a,b\in H^{p,q}(X)$. So by the hard Lefschetz theorem on $Y$, the Hermitian form $\left<\alpha,\beta\right>=\int_X\alpha\wedge\bar{\beta}\wedge h_1\cdots h_k$ on $H^{p,q}(X)$ is nondegenerate, proving the statement.
\end{proof}

We see that the mixed Lefschetz theorem for very ample classes $h_1,\cdots,h_k$ follows from the classical Lefschetz theorem for $h_k$ on $Y$. To generalize this idea (for example, to deal with ample classes that may not be rational), we will need a quantitative version. The idea is to investigate the metric induced by $h_k$ on $Y$. We have the following key observation:

\textbf{Observation 3:} Let $v,w\in\Amp(X)$. Suppose that $v$ is very ample and is represented by a smooth divisor $Y$. Then for $a\in H^{p,q}(X)$, $\|a|_Y\|_w$ depends only on $v$, not on $Y$.

This motivates the definition of the metric $\|\cdot\|_{(v,w)}$ below.

\textbf{The metric $\|\cdot\|_{(v,w)}$}

In sections 4-5 we will define an Hermitian metric $\|\cdot\|_{(v,w)}$ on $H^{p,q}(X)$ with $p+q\leq n-r$ or $p+q\geq n+r$, where $v=w_1\cdots w_r\in H^{r,r}(X,\R)$ is the product of $r$ K\"ahler classes and $w$ is a K\"ahler class. It satisfies the following property. If $w_1,\cdots,w_r$ happen to be very ample and $Y$ is a complete intersection of smooth divisors representing them, then $\|a\|_{(v,w)}=\|a|_Y\|_w$ for $a\in H^{p,q}(X)$ with $p+q\leq n-r$, while $\|va\|_{(v,w)}=\|a|_Y\|_w$ for $a\in H^{p,q}(X)$ with $p+q\geq n-r$ (lemma \ref{lemma: restriction identities}).

A local version $|\cdot|_{(\nu,\omega)}$ on $\bigwedge^{p,q}V$, where $V$ is a finite dimensional vector space, will also be defined.

\textbf{Local-to-global part}

The local-to-global part is done in section 5. The original motivation of section 5 is to give a more natural proof of the Lefschetz theorem for mixed K\"ahler classes (the main result of \cite{DN2}). Recall that the classical Lefschetz theorem is proven by representing the cohomology groups $H^{p,q}(X)$ by the spaces of harmonic forms $\H^{p,q}_\omega(X)$. Then the Lefschetz decompositions, the Hodge-Riemann relations, etc., are reduced to local statements. Here we follow the same strategy and prove the existence of harmonic forms with respect to $(\nu,\omega)$, where $\nu$ is a product of K\"ahler forms. In particular we will prove the following equation which relates the local norm $|\cdot|_{(\nu,\omega)}$ with the global norm $\|\cdot\|_{\nu,\omega}$:
\begin{align}\label{global norm = local norm introduction}
    \|a\|_{(\nu,\omega)}^2=\inf_{[\alpha]=a}\int_X|\alpha|_{(\nu,\omega)}^2\frac{\omega^n}{n!}
\end{align}
and the infimum of RHS is achieved by a unique harmonic form $\alpha\in\H^{p,q}_{(\nu,\omega)}(X)$. The harmonic spaces $\H^{p,q}_{(\nu,\omega)}(X)$ are preserved by the operator $L=-\wedge\omega$ and we have the Lefschetz decomposition
\[\mathcal{H}^{p,q}(X)=\bigoplus_{i\geq 0}L^i\mathcal{P}^{p-i,q-i}(X).\]
This would immediately imply the main result of \cite{DN2}. Moreover, one can derive the following proposition which will be used in the proof of theorem \ref{theorem 2}.
\begin{proposition}\label{proposition: pointwise zero}(= corollary \ref{corollary: pointwise zero})
    Let $\omega_1,\cdots,\omega_r$ be K\"ahler forms on a compact K\"ahlerian manifold $X$. Suppose that $a\in H^{p,q}(X)$ satisfies $a\cdot[\omega_1\cdots\omega_r]=0$, then there exists a $d$-closed $\alpha\in\A^{p,q}(X)$ repesenting $a$ such that $\alpha\cdot\omega_1\cdots\omega_r=0$. 
\end{proposition}

\textbf{Local part}

Section 6 is devoted to the proof of the following two local theorems:
\begin{proposition}\label{introduction: local 1}(= proposition \ref{proposition: local 1})
    Let $V$ be an $n$ dimensional $\C$-vector space, $\omega_1,\cdots,\omega_r,\omega$ positive (1,1)-forms on $V$. Let $\nu=\omega_1\cdots\omega_r$. Suppose that $\omega_i\leq N\omega\ (i=1,\cdots,r)$. Then for all $\alpha\in\bigwedge^{p,q}V$ with $p+q\leq n-r$, we have
    \begin{align*}
        |\nu\alpha|_\omega\overset{N}{\lesssim}|\alpha|_{(\nu,\omega)}
    \end{align*}
\end{proposition}
\begin{proposition}\label{introduction: local 2}($\subset$ proposition \ref{proposition: local 2})
    Let $V,W$ be $\C$-vector spaces of dimension $n,m$ respectively and $\pi:V\rightarrow W$ a linear map. Let $\omega_1,\cdots,\omega_r$ be positive (1,1)-forms on $W$ and $\omega$ a positive (1,1)-form on $V$. Let $\nu=\pi^*\omega_1\cdots\pi^*\omega_r$. For $\epsilon>0$, define $\nu_{\epsilon}=(\pi^*\omega_1+\epsilon\omega)\cdots(\pi^*\omega_s+\epsilon\omega)$. Let $\alpha\in\bigwedge^{p,q}V$ with $p+q\geq n-r$. Suppose that $\nu\alpha=0$, then
    \begin{align*}
        \lim_{\epsilon\rightarrow0}|\nu_\epsilon\alpha|_{(\nu_\epsilon,\omega)}=0.
    \end{align*}
\end{proposition}
Assuming these results, we can now sketch the proof of theorem \ref{theorem 2}.

\textbf{Sketch of the proof of theorem \ref{theorem 2}}
We first describe the basic idea. The classes $\pi^*w_1,\cdots,\pi^*w_s$ are not strictly positive, so we cannot apply the techniques developed in section 5 directly. The idea is to perturb $\pi^*w_1,\cdots,\pi^*w_s$ by adding a small positive class. So we will consider $v_\epsilon=(\pi^*w_1+\epsilon w)\cdots(\pi^*w_s+\epsilon w)$ and $u_\epsilon=v_\epsilon\cdot w_{s+1}\cdots w_r$, where $w$ is an arbitrary K\"ahler class on $Y$. Given $a\in H^{p,q}(Y)$, we want to show that $au=0$ implies $av=0$. The proof consists of two steps:
\begin{itemize}
    \item Step 1: $au=0$ implies $\|au_\epsilon\|_{(u_\epsilon,w)}$ small.
    \item Step 2: $\|au_\epsilon\|_{(u_\epsilon,w)}$ small implies $\|av_\epsilon\|_w$ small.
\end{itemize}
Letting $\epsilon\rightarrow 0$ gives $av=0$. These two steps make use of the two local propositions respectively.

Let's describe how these two steps are done. For simplicity we consider the case $r=s+1$ (this is in fact all we need in application to theorem \ref{theorem 3}). Then we can simplify notation by writing $w$ for $w_r$. Take a K\"ahler form $\omega$ on $Y$ representing $w$, and take K\"ahler forms $\omega_1,\cdots,\omega_s$ on $X$ representing $w_1,\cdots,w_s$. Let $\nu=\pi^*\omega_1\cdots\pi^*\omega_s,\ \nu_{\epsilon}=(\pi^*\omega_1+\epsilon\omega)\cdots(\pi^*\omega_s+\epsilon\omega),\ \mu_\epsilon=\nu_{\epsilon}\omega$.

Step 1: Suppose $a\in H^{p,q}(Y)$ satisfies $au=0$. We want to prove that $\lim_{\epsilon\rightarrow0}\|a\cdot u_\epsilon\|_{(u_\epsilon,w)}=0$. Let $b=aw$. Then $bv=0$. We claim that there exits $\beta\in\A^{p+1,q+1}(Y)$ representing $b$ such that $\beta\nu=0$. This is the only place where we have used the assumption that $H^*(Y)$ is a free $H^*(X)$-module. The proof of this claim is a simple application of proposition \ref{proposition: pointwise zero} above.

We have strengthen the identity $bv=0$ to the identity $\beta\nu$ which holds on the level of differential forms. This enables us to `pass to local'. By proposition \ref{introduction: local 2} above  (a local result), we have pointwisely $\lim_{\epsilon\rightarrow0}|\nu_{\epsilon}\beta|_{(\nu_{\epsilon},\omega)}=0$. By a basic property of the metric we have $|\nu_{\epsilon}\beta|_{(\nu_{\epsilon},\omega)}\approx|\nu_{\epsilon}\beta|_{(\mu_\epsilon,\omega)}$. So integrating over $Y$ we get $\lim_{\epsilon\rightarrow0}\|a\cdot u_\epsilon\|_{(u_\epsilon,w)}=0$.

Step 2: By a basic property of the metric we have $\|a\|_{(u_\epsilon,w)}=\|a\cdot u_\epsilon\|_{(u_\epsilon,w)}$, so step 1 implies $\lim_{\epsilon\rightarrow0}\|a\|_{(u_\epsilon,w)}=0$. We represent $a$ by harmonic forms $\alpha_\epsilon\in \H^{p,q}_{(\mu_\epsilon,\omega)}(Y)$. Then $\lim_{\epsilon\rightarrow0}\int_Y|\alpha_\epsilon|^2_{(\mu_\epsilon,\omega)}\omega^n=0$.
This again enables us to `pass to local'. An application of proposition \ref{introduction: local 1} (another local result) soon implies that $\lim_{\epsilon\rightarrow 0}\|a\cdot v_\epsilon\|_w=0$.

\subsection{Notations}
We usually denote by $X$ a compact K\"ahlerian manifold or a smooth projective variety over $\C$ of dimension $n$. For $p,q$ integers, $H^{p,q}(X)=H^{p,q}(X,\C)=\frac{d\text{-closed }C^{\infty}(p,q)\text{-forms}}{d\text{-exact }C^{\infty}(p,q)\text{-forms}}$, $H^{p,p}(X,\R)=H^{p,p}(X)\cap H^{2p}(X,\R)$. $\mathrm{Kah}(X)\subset H^{1,1}(X,\R)$ denotes the K\"ahler cone. If $X$ is projective, $N^1(X)=\mathrm{Pic}(X)/\equiv$ denotes the group of numerical equivalence classes of divisors on $X$, $\mathrm{Amp}(X)\subset N^1(X)_\R$ denotes the ample cone. We consider $N^1(X)_\R$ as a subspace of $H^{1,1}(X,\R)$, so $\mathrm{Amp}(X)=\mathrm{Kah}(X)\cap N^1(X)_\R$. $\A^{p,q}(X)$ denotes the space of $C^\infty\ (p,q)$-forms on $X$. We find it necessary to distinguish a cohomology class and the differential forms representing it. Elements of $H^{p,q}(X)$ are usually denoted by $a,b,v,w$, while elements of $\A^{p,q}(X)$ are usually denoted by Greek letters $\alpha,\beta,\nu,\omega\cdots$. $\omega$ usually mean a K\"ahler form, while $w,h$ usually mean K\"ahler or ample classes. If we want to talk about elements of $\bigwedge^{p,q}V$ where $V$ is a finite dimensional vector space, we prefer to use Greek letters.

Let $E=E_0\left<h\right>$ be an $\R$-twisted vector bundle, where $E_0$ is a vector bundle and $h\in N^1(X)_\R$. For any $h'\in N^1(X)_\R$ we write $E\left<h'\right>$ for the $\R$-twisted vector bundle $E_0\left<h+h'\right>$. $\P(E)$ is the same as $\P(E_0)$ (the bundle of 1-dimensional quotients of $E_0$) but the universal quotient bundle $\O_{\P(E)}(1)$ is equal to $\O_{\P(E_0)}(1)\left<\pi^*h\right>$. If $X\rightarrow Y$ is a morphism, $h\in N^1(Y)$, $E$ an $\R$-twisted vector bundle on $X$, we sometimes abbreviate $E\left<\pi^*h\right>$ to $E\left<h\right>$.

If $Y\subset X$ is a subvariety, $h\in\Amp(X)$, we usually abbreviate $h|_Y$ to $h$. So the ugly notation $\|a|_Y\|_{h|_Y}$ would be replaced by $\|a|_Y\|_h$.

For non-negative real numbers $a,b$, we write $a\lesssim b$ (resp. $a\gtrsim b$) iff there is a constant $C>0$ depending only on the dimension $n$ of $X$ such that $a\leq Cb$ (resp. $Ca\geq b$). We write $a\approx b$ iff $a\lesssim b$ and $a\gtrsim b$. Let $N\geq 2$ be a real number, We write $a\overset{N}{\lesssim}b$ (resp. $a\overset{N}{\gtrsim}b$) iff there is a constant $C>0$ depending only on the dimension $n$ such that $a\leq N^Cb$ (resp. $a\geq N^{-C}b$). We write $a\overset{N}{\approx}b$ iff $a\overset{N}{\lesssim}b$ and $a\overset{N}{\gtrsim}b$.

\subsection*{Acknowledgments}
We would like to thank professor Weizhe Zheng, who initiated the research and taught us much about the Hodge-Riemann relations, ample vector bundles, etc. In fact both theorem \ref{theorem 1} and \ref{theorem 3} were conjectured by him. We would like to thank Enhan Li, Haofeng Zhang for many enlightening discussions during a summer school.

\section{Preliminaries on metrics}
In this section we summarize some known results in the form that we are going to use, and prove some simple lemmas for later reference.

Let $X$ be a compact K\"ahlerian manifold of dimension $n$, $\omega$ a K\"ahler form on $X$. The metric on $\bigwedge^{p,q}T^*X$ induced by $\omega$ is defined as follow. For every $x\in X$ we can choose a local frame $dz^1,\cdots,dz^n$ so that $\omega_x=\sqrt{-1}\sum_{i=1}^n dz^i\wedge d\bar{z}^i\big|_x$. If $\alpha_x\in\bigwedge^{p,q}T^*_xX$, we can write $\alpha_x=\sum_{I,J}\alpha_{IJ}dz^I\wedge d\bar{z}^J\big|_x$ and define $|\alpha_x|_{\omega}^2=\sum_{I,J}|\alpha_{IJx}|^2$. For $\alpha\in\mathcal{A}^{p,q}(X)$ we define $\|\alpha\|_\omega^2:=\int_X|\alpha|_\omega^2\frac{\omega^n}{n!}$.

Every class $[\alpha]\in H^{p,q}(X)=\frac{d\text{-closed }(p,q)\text{-forms}}{d\text{-exact }(p,q)\text{-forms}}$ contains a unique minimizer of the norm $\|\alpha\|_\omega^2$, denoted by $\alpha_H$. These are called harmonic forms. The space of harmonic forms is denoted by $\mathcal{H}^{p,q}(X)\subset\mathcal{A}^{p,q}(X)$. We have $\mathcal{H}^{p,q}(X)\cong H^{p,q}(X)$. We put the norm $\|\cdot\|_\omega$ on $H^{p,q}(X)$ via this isomorphism, i.e., $\|[\alpha]\|_\omega:=\|\alpha_H\|_\omega$.

We explain why the metric $\|\cdot\|_\omega^2$ on $H^{*,*}(X)$ depends only on the cohomology ring.

\textbf{Computing the metric purely cohomologically.}

Denote the operator $-\wedge\omega:H^{p,q}(X)\rightarrow H^{p+1,q+1}(X)$ by $L$. For $p+q\leq n$, let $k=n-p-q$, the hard-Lefschetz theorem states that
\begin{align}\label{Lefschetz isomorphism}
    L^k:H^{p,q}(X)\rightarrow H^{n-q,n-p}(X)
\end{align}
is an isomorphism. Define $P^{p,q}(X)=\ker\left(L^{k+1}:H^{p,q}(X)\rightarrow H^{n-q+1,n-p+1}(X)\right)$. Then we have the Lefschetz decomposition
\[H^{p,q}(X)=\bigoplus_{i=0}^{\min(p,q)}L^i P^{p-i,q-i}(X),\quad(p+q
\leq n)\]
which is an orthogonal decomposition. Consider the following Hermitian form on $H^{p,q}(X)$:
\[\left<a,b\right>_{\omega^k}:=\sqrt{-1}^{(p+q)^2}(-1)^q\int_Xa\bar{b}\cdot\omega^k.\]
The metric $\|\cdot\|_{\omega}$ can be computed as follow. If $p+q\leq n$, denote $k=n-p-q$. Every $a\in H^{p,q}(X)$ can be written uniquely as $\alpha=\sum_{i=0}^{\min(p,q)}[\omega]^ia_i$ with $a_i\in P^{p-i,q-i}(X)$, then
\begin{align}\label{compute the metric by Q1}
    \|a\|_\omega^2=\sum_{i=0}^{\min(p,q)}\frac{i!}{(k+i)!}\left<a_i,a_i\right>_{\omega^{k+2i}}.
\end{align}
If $a\in H^{n-q,n-p}$ with $p,q\geq0,\ p+q\leq n$, still let $k=n-p-q$. Using the decomposition $H^{n-q,n-p}(X)=\bigoplus_{i=0}^{\min(p,q)}L^{k+i} P^{p-i,q-i}(X)$, we can write $a$ uniquely as $a=\sum_{i=0}^{\min(p,q)}[\omega]^{k+i}a_i$ with $a_i\in P^{p-i,q-i}(X)$, then
\begin{align}\label{compute the metric by Q2}
    \|a\|_\omega^2=\sum_{i=0}^{\min(p,q)}\frac{(k+i)!}{i!}\left<a_i,a_i\right>_{\omega^{k+2i}}.
\end{align}
 In particular, we see that the isomorphism (\ref{Lefschetz isomorphism}) is a quasi-isometry, in the sense of the following equation:
\begin{align}\label{isometry0}
    \|\alpha\|_{\omega}\approx\|\omega^k\alpha\|_{\omega},\quad(\forall\alpha\in H^{p,q}(X),\ p+q+k=n).
\end{align}

\textbf{Restriction to hyperplane sections}

\begin{lemma}\label{lemma: restriction quasi-isometry 1}
     Let $X$ be a compact K\"ahler manifold of dimension $n$, $h\in \Kah(X)$. Then
     \begin{align}\label{isometry1}
         \|a\|_h\approx\|ha\|_h\quad(\forall a\in H^{p,q}(X),\ p+q\leq n-1).
     \end{align}
     If $h$ is very ample and $Y$ is a smooth divisor representing $h$, then
     \begin{align}\label{isometry2}
         \|ha\|_h\approx\|a|_Y\|_{h}\quad(\forall a\in H^{p,q}(X),\forall p,q\geq 0).
     \end{align}
\end{lemma}
\begin{proof}
    (\ref{isometry1}) is obvious by the formulas (\ref{compute the metric by Q1}) and (\ref{compute the metric by Q2}). The proof of (\ref{isometry2}) is based on the following observation:
    \[\int_Xah=\int_Ya\]
    for all $a\in H^{n-1,n-1}(X)$. As a consequence,
    \[\left<a,b\right>_{h^k}=\left<a|_Y,b|_Y\right>_{h^{k-1}}\]
    for all $a,b\in H^{p,q}(X)$ with $p+q\leq n-1,\ k=n-p-q$. We see that the natural restriction map $H^{p,q}(X)\rightarrow H^{p,q}(Y)$ takes $P^{p,q}(X)$ isomorphically on to $P^{p,q}(Y)$ if $p+q\leq n-2$ and takes $P^{p,q}(X)$ into a subspace of $P^{p,q}(Y)$ if $p+q=n-1$ (since $P^{p,q}$ is nothing but the orthogonal complement of $h H^{p-1,q-1}$ in $H^{p,q}$ with respect to $\left<-,-\right>$). In particular the natural restriction map $H^{p,q}(X)\rightarrow H^{p,q}(Y)$ preserves the Lefschetz decomposition if $p+q\leq n-1$. So by the formula (\ref{compute the metric by Q1}) we have
    \begin{align*}
        \|a\|_h\approx\|a|_Y\|_h
    \end{align*}
    for $a\in H^{p,q}(X)$ with $p+q\leq n-1$. Together with (\ref{isometry1}) this implies (\ref{isometry2}) in the case $p+q\leq n-1$, while the case $p+q\geq n$ of (\ref{isometry2}) follows from this and equation (\ref{isometry0}).
\end{proof}

\textbf{Rescaling}
\begin{align}\label{rescaling}
    \|\cdot\|_{\lambda\omega}=\lambda^{\frac{n-p-q}{2}}\|\cdot\|_\omega
\end{align}

\textbf{Comparing two metrics}

Suppose that $H_1,H_2$ are two Hermitian metrics on a finite dimenisional $\C$-vector space $V$ satisfying $H_1\geq H_2$, then for any $\alpha\in\bigwedge^{p,q}V$ we have
\begin{align}\label{compare3}
    \|\alpha\|_{H_1}\leq\|\alpha\|_{H_2}.
\end{align}
More generally, if $V_1,V_2$ are two finite dimensional $\C$-vector spaces and $f:V_1\rightarrow V_2$ is a linear map, $H_1$ and $H_2$ are Hermitian metrics on $V_1,V_2$ respectively satisfying $H_1\geq f^*H_2$. Then for all $\alpha\in\bigwedge^{p,q}V_2$ we have
\begin{align}\label{compare4}
    \|f^*\alpha\|_{H_1}\leq\|\alpha\|_{H_2}
\end{align}
 Suppose that $\omega_1,\omega_2$ are two K\"ahler forms satisfying $N^{-1}\omega_1\leq \omega_2\leq N\omega_1$ for some $N>1$ then
 \begin{align}\label{compare1}
     \|\alpha\|_{\omega_1}\overset{N}{\approx}\|\alpha\|_{\omega_2}\quad(\forall\alpha\in H^{p,q}(X),\forall p,q\geq 0).
 \end{align}
 (Indeed, the assumption $N^{-1}\omega_1\leq\omega_2$ implies the pointwise inequality $N^{p+q}|\alpha|_{\omega_1}^2\geq|\alpha|^2_{\omega_2}$, while $\omega_2\leq N\omega_1$ gives $\omega_2^n\leq N^n\omega_1^n$. Integrating over $X$ gives $\|\alpha\|_{\omega_2}\leq N^{(p+q+n)/2}\|\alpha\|_{\omega_1}.$)
 
If $h_1,h_2\in\mathrm{Kah}(X)$ satisfies $Nh_1-h_2,Nh_2-h_1\in\mathrm{Kah}(X)$, then we can choose K\"ahler forms $\tau_1,\tau_2$ such that $[\tau_1]=Nh_1-h_2,\ [\tau_2]=Nh_2-h_1$. Put $\omega_1=\frac{N\tau_1+\tau_2}{N^2-1},\ \omega_2=\frac{N\tau_2+\tau_1}{N^2-1}$, we see that $[\omega_1]=h_1,\ [\omega_2]=h_2$ and $N^{-1}\omega_1\leq\omega_2\leq N\omega_1$. So applying the previous discussion we get
\begin{align}\label{compare2}
    \|\alpha\|_{h_1}\overset{N}{\approx}\|\alpha\|_{h_2}\quad(\forall\alpha\in H^{p,q}(X),\forall p,q\geq 0).
\end{align}

\textbf{Two simple inequalities}

Let $\alpha\in\A^{p,q}(X)$, then
\begin{align}\label{eq30}
    \|\alpha\|_{\omega}^2\leq\mathrm{Vol}_{\omega}(X)\cdot\sup_X|\alpha|_{\omega}^2
\end{align}
where $\mathrm{Vol}_{\omega}(X)=\int_X\frac{\omega^n}{n!}$.

Let $a\in H^{p,q}(X),\ b\in H^{n-p.n-q}(X)$, then
\begin{align}\label{eq300}
    \|a\|_{\omega}\cdot\|b\|_{\omega}\geq\left|\int_Xab\right|.
\end{align}
(\ref{eq30}) is trivial. To prove (\ref{eq300}), take harmonic representatives $\alpha,\beta$ of $a,b$ respectively, then
\begin{align*}
    \begin{split}
        \left|\int_Xab\right|=\left|\int_X\alpha\wedge\beta\right|\leq\int_X|\alpha\wedge\beta|_\omega\frac{\omega^n}{n!}\leq&\int_X|\alpha|_\omega|\beta|_\omega\frac{\omega^n}{n!}\\\leq\|\alpha\|_\omega\|\beta\|_\omega=\|a\|_{\omega}\|b\|_{\omega},
    \end{split}
\end{align*}
which proves (\ref{eq300}).

We end this section by a lemma that is frequently used in this paper.
\begin{lemma}\label{lemma: frequently used}
    Let $X$ be a compact K\"ahlerian manifold and $Y\subset X$ a smooth submanifold. Let $w\in\Kah(X)$. Then
    \[\|a|_Y\|_w^2\leq C\|a\|_w^2\vol_w(Y)\]
    for all $a\in H^*(X)$, where $C$ is a constant depending only on $X$ and $w$.
\end{lemma}
\begin{proof}
    Take a K\"ahler metric $\omega$ on $X$ representing $w$. Let $\alpha\in\H^*_\omega(X)$ be the harmonic form representing $a$. Then
    \[\|a|_Y\|_w^2\leq\|\alpha|_Y\|_\omega^2\leq\sup_X|\alpha|_\omega^2\vol_\omega(Y).\]
    Since $\H^*_\omega(X)$ is finite dimensional, any two norms on a finite dimensional $\C$-vector space are equivalent. So there exists $C>0$ depending only on $X,\omega$ such that $\sup_X|\alpha|_\omega\leq C\|\alpha\|_\omega$. This proves the lemma.
\end{proof}

\section{Proof of theorem \ref{theorem 1}}
We first prove proposition \ref{proposition 1}
\begin{proof}(of proposition \ref{proposition 1})
    Let $\alpha\in\H_{\omega_0}^{p,q}(X)$ be an arbitrary harmonic $(p,q)$-form. We need to prove that
    \begin{align}\label{eq34}
        \|c_e\cdot[\alpha]\|_{\omega_0}\geq c\|\alpha\|_{\omega_0}
    \end{align}
    for some constant $c>0$ to be determined. Since $N^1(X)_{\Q}$ is dense in $N^1(X)_\R$, in order to prove (\ref{eq34}), by continuity we may assume that $E_1,\cdots,E_k$ are $\Q$-twisted. Then it is well known that there exists a finite dominant morphism $\pi:Y\rightarrow X$, with $Y$ smooth, such that $\pi^*E_1,\cdots,\pi^*E_k$ are vector bundles (c.f.\cite{BG} lemma 2.1). Let $d=\deg(Y/X)$. Let $E=\pi^*E_1\oplus\cdots\oplus\pi^*E_k,\ e=e_1+\cdots+e_k$. Use the notations $P,\varphi,\tilde{\pi}$ indicated in the following diagram:
\[\begin{tikzcd}
	{P=\mathbb{P}(E)} & Y \\
	& X
	\arrow["\varphi", from=1-1, to=1-2]
	\arrow["{\tilde{\pi}}"', from=1-1, to=2-2]
	\arrow["\pi", from=1-2, to=2-2]
\end{tikzcd}\]
As indicated in the introduction, the idea is to make the proof of the Bloch-Giesecker's theorem quantitative (applied to the vector bundle $E\left<\omega_0\right>$). Recall that the Bloch-Giesecker's theorem is proven by applying the hard Lefschetz theorem to $\O_{\P(E\left<\omega_0\right>)}(1)$ on $P$. Let $u=c_1(\O_{\P(E)}(1))+\tilde{\pi}^*[\omega_0](=c_1(\O_{\P(E\left<\omega_0\right>)}(1)))$. Then
\begin{align}\label{eq32}
    u^e-\varphi^*c_1(E\left<\omega_0\right>)u^{e-1}+\cdots+(-1)^e\varphi^*c_e(E\left<\omega_0\right>)=0.
\end{align}
Recalling the definition of $c_e$ from the statement of the proposition, we get
\begin{align}\label{eq33}
    \tilde{\pi}^*c_e=\varphi^*c_e(E\left<\omega_0\right>)=v\cdot u,
\end{align}
where we have used (\ref{eq32}) in the second equality, and
\[v:=(-1)^{e-1}u^{e-1}+(-1)^{e-2}\varphi^*c_1(E\left<\omega_0\right>)u^{e-2}+\cdots+\varphi^*c_{e-1}(E\left<\omega_0\right>).\]
So
\begin{align}\label{eq35}
    \tilde{\pi}^*[\alpha]\cdot v\cdot u=\tilde{\pi}^*([\alpha]\cdot c_e).
\end{align}
In order to make the hard Lefschetz theorem quantitative, we need to choose a K\"ahler metric on $P$ representing $u$. Let $\omega$ be a K\"ahler form on $P$ representing $c_1(\O_{\P(E)}(1))$. Then $\tilde{\omega}:=\omega+\tilde{\pi}^*\omega_0$ would be a K\"ahler form on $P$ representing $u$. By (\ref{isometry0}) (which says that the Lefschetz operator $-\wedge u=-\wedge\tilde{\omega}$ is a quasi-isometry) and (\ref{eq35}) above we get
\begin{align}\label{eq36}
    \|\tilde{\pi}^*[\alpha]\cdot v\|_{\tilde{\omega}}\approx\|\tilde{\pi}^*([\alpha]\cdot c_e)\|_{\tilde{\omega}}
\end{align}
Recall that our goal is to bound $\|[\alpha]\cdot c_e\|_{\omega_0}$ from below. So in light of (\ref{eq36}), we need to bound $\|\tilde{\pi}^*[\alpha]\cdot v\|_{\tilde{\omega}}$ from below. The trick is to find a differential form $\beta$ with norm bounded from above so that $\int_P\tilde{\pi}^*\alpha\cdot v\cdot\beta$ is bounded from below. We take
\[\beta=*\alpha\in\H_{\omega_0}^{n-q,n-p}(X).\]
Then
\begin{align*}
    \int_X\alpha\wedge\bar{\beta}=\|\alpha\|_{\omega_0}^2.
\end{align*}
\begin{align*}
    (-1)^{e-1}\int_P\tilde{\pi}^*\alpha\cdot\tilde{\pi}^*\bar{\beta}\cdot v=\int_Y\pi^*\alpha\cdot\pi^*\bar{\beta}=d\int_X\alpha\cdot\bar{\beta}=d\|\alpha\|_{\omega_0}^2.
\end{align*}
So by (\ref{eq300}),
\begin{align}\label{eq37}
    \|\tilde{\pi}^*[\alpha]\cdot v\|_{\tilde{\omega}}\cdot\|[\beta]\|_{\tilde{\omega}}\geq\left|\int_P\tilde{\pi}^*\alpha\cdot v\cdot\bar{\beta}\right|=d\|\alpha\|_{\omega_0}^2.
\end{align}
Combining (\ref{eq36}) and (\ref{eq37}) we get
\begin{align}\label{eq38}
    \|\tilde{\pi}^*([\alpha]\cdot c_e)\|_{\tilde{\omega}}\cdot\|\tilde{\pi}^*\beta\|_{\tilde{\omega}}\geq d\|\alpha\|_{\omega_0}^2.
\end{align}
By (\ref{eq30}), this implies
\begin{align}\label{eq39}
    \mathrm{Vol}_{\tilde{\omega}}(P)\cdot\sup_P|\tilde{\pi}^*(([\alpha]\cdot c_e)_H)|_{\tilde{\omega}}\cdot\sup_P|\tilde{\pi}^*\beta|_{\tilde{\omega}}\geq d\|\alpha\|_{\omega_0}^2,
\end{align}
where we have used the simple observation that $\tilde{\pi}^*(([\alpha]\cdot c_e)_H)$ is a form representing the class $\tilde{\pi}^*([\alpha]\cdot c_e)$, so $\|\tilde{\pi}^*(([\alpha]\cdot c_e)_H)\|_{\tilde{\omega}}\geq\|\tilde{\pi}^*([\alpha]\cdot c_e)\|_{\tilde{\omega}}$.

Now by (\ref{compare4}), since $\tilde{\omega}\geq\tilde{\pi}^*\omega_0$, we get
\[\sup_P|\tilde{\pi}^*(([\alpha]\cdot c_e)_H)|_{\tilde{\omega}}\leq\sup_X|([\alpha]\cdot c_e)_H|_{\omega_0},\quad\sup_P|\tilde{\pi}^*\beta|_{\tilde{\omega}}\leq\sup_X|\beta|_{\omega_0}.\]
Plug them into (\ref{eq39}) we get
\begin{align}\label{eq310}
    \mathrm{Vol}_{\tilde{\omega}}(P)\cdot\sup_X|([\alpha]\cdot c_e)_H|_{\omega_0}\cdot\sup_X|\beta|_{\omega_0}\geq d\|\alpha\|_{\omega_0}^2.
\end{align}

The next key observation is that $\mathrm{Vol}_{\tilde{\omega}}(P)$ is purely topological (independent of the choice of $\omega$). We compute it as follow:
\begin{align*}
    (n+e-1)!\cdot\mathrm{Vol}_{\tilde{\omega}}(P)=&\int_P\tilde{\omega}^{n+e-1}\\
    =&\int_Pc_1\left(\O_{\P(E\left<\omega_0\right>)}(1)\right)^{n+e-1}\\
    =&\int_Ys_n(c(E\left<\omega_0\right>))\\
    =&d\int_Xs_n(c(E_1\left<\omega_0\right>)\cdots c(E_k\left<\omega_0\right>))
\end{align*}
where $s_n$ is a polynomial of the Chern classes depending only on $n,e$. So we get
\begin{align}
    \mathrm{Vol}_{\tilde{\omega}}(P)\leq d\cdot C(M,\omega_0)
\end{align}
where $C(M,\omega_0)$ is a constant depending only on $M,X,\omega_0$. Plug this into (\ref{eq310}), we get
\begin{align}\label{eq311}
    C(M,\omega_0)\cdot\sup_X|([\alpha]\cdot c_e)_H|_{\omega_0}\cdot\sup_X|\beta|_{\omega_0}\geq \|\alpha\|_{\omega_0}^2.
\end{align}
Since $\H^{n-q,n-p}_{\omega_0}(X)$ is finite dimensional, and any two norms on a finite dimensional $\C$-vector space are equivalent, there exists a constant $C_0>0$ depending only on $X$ and $\omega_0$ such that $\sup_{X}|\eta|_{\omega_0}\leq C_0\|\eta\|_{\omega_0}$ for all $\eta\in\H^{n-q,n-p}_{\omega_0}(X)$. So (\ref{eq311}) implies
\begin{align}
    C(M,\omega_0)C_0^2\|[\alpha]\cdot c_e\|_{\omega_0}\|\beta\|_{\omega_0}\geq \|\alpha\|_{\omega_0}^2.
\end{align}
Note that $\|\beta\|_{\omega_0}=\|*\alpha\|_{\omega_0}=\|\alpha\|_{\omega_0}$. So we get
\begin{align}
    C(M,\omega_0)C_0^2\|[\alpha]\cdot c_e\|_{\omega_0}\geq \|\alpha\|_{\omega_0}.
\end{align}
This proves (\ref{eq34}) with $c=C(M,\omega_0)^{-1}C_0^{-2}$.
\end{proof}
Now we prove theorem \ref{theorem 1 HR} (which includes theorem \ref{theorem 1})
\begin{proof}(of theorem \ref{theorem 1 HR})
    The hard Lefschetz property is an immediate consiquence of proposition \ref{proposition 1}. To prove the Hodge-Riemann relation, for every $t\geq 0$ consider the following Hermitian form on $H^{p,q}(X)$:
    \[\left<a,b\right>_t=(-1)^q\sqrt{-1}^{(p+q)^2}\int_Xa\bar{b}\cdot c_{e_1}(E_1\left<th\right>)\cdots c_{e_k}(E_k\left<th\right>).\]
    The hard Lefschetz property applied to the $\R$-twisted vector bundles $E_1\left<th\right>,\cdots,E_k\left<th\right>$ (together with Serre duality) implies that $\left<-,-\right>_t$ is nondegenerate for every $t\geq 0$. In addition, the hard Lefschetz property applied to the $\R$-twisted vector bundles $E_1\left<th\right>,\cdots,E_k\left<th\right>,h,h$ implies that $\left<-,-\right>_t$ is also nondegenerate on the subspace $hH^{p-1,q-1}(X)$. By continuity the index of these Hermitian forms on $H^{p,q}(X)$ and on the subspace $hH^{p-1,q-1}(X)$ are independent of $t$. So their index on the orthogonal complement of $hH^{p-1,q-1}(X)$ is also independent of $t$. Letting $t\rightarrow\infty$ we see from the classical Hodge Riemann relation that the Hermitian forms are positive definite on the orthogonal complement of $hH^{p-1,q-1}(X)$ for large $t$. So for $t=0$, $\left<-,-\right>$ is also positive definite on the orthogonal complement of $hH^{p-1,q-1}(X)$, as desired.
\end{proof}

 \section{The metric $\|\cdot\|_{(v,w)}$}
 
 \begin{lemma}\label{lemma: restriction quasi-isometry 2}
     Let $X$ be a smooth projective variety of dimension $n$ over $\C$, $h_0,\cdots,h_r\in\mathrm{Amp}(X)\ (r\leq n)$ such that $h_1,\cdots,h_r$ are very ample. Let $Y_1,\cdots,Y_r$ be smooth divisors representing $h_1,\cdots,h_r$ that intersect transversally. Let $v=h_1\cdots h_r\in H^{r,r}(X,\R)$, $Y=Y_1\cap\cdots\cap Y_r$. Suppose that $N\geq 2$ is a real number such that $Nh_i-h_j\in\mathrm{Amp}(X)$ for all $0\leq i,j\leq r$, then
     \begin{align}\label{isometry3}
         \|a|_{Y}\|_{h_0}\overset{N}{\approx}\|va\|_{h_0}\quad(\forall a\in H^{p,q}(X),\forall p,q\geq 0),
     \end{align}
     \begin{align}\label{isometry4}
         \|a\|_{h_0}\overset{N}{\approx}\|va\|_{h_0}\quad(\forall a\in H^{p,q}(X),\ p+q\leq n-r).
     \end{align}
 \end{lemma}
 \begin{proof}
     Use induction on $r$. When $r=0$ the assertion is empty. Now suppose $k>0$. Applying the induction hypothesis to $h_0|_{Y_r},\cdots,h_{r-1}|_{Y_r},a|_{Y_r}$, we get
     \begin{align}\label{eq11}
         \|a|_Y\|_{h_0}\overset{N}{\approx}\|h_1\cdots h_{r-1}a|_{Y_r}\|_{h_0},
     \end{align}
     \begin{align}\label{eq12}
         \|a|_Y\|_{h_0}\overset{N}{\approx}\|a|_{Y_r}\|_{h_0}\quad(p+q\geq n-r).
     \end{align}
     By (\ref{isometry1})(\ref{isometry2}) we have
     \begin{align}\label{eq13}
         \|h_1\cdots h_{r-1}a|_{Y_r}\|_{h_r}\approx\|h_1\cdots h_ra\|_{h_r},
     \end{align}
     \begin{align}\label{eq14}
         \|a|_{Y_r}\|_{h_r}\approx\|a\|_{h_r}\quad(p+q\geq n-1).
     \end{align}
     By (\ref{compare2}) we have
     \begin{align}\label{eq15}
         \|h_1\cdots h_{r-1}a|_{Y_r}\|_{h_0}\overset{N}{\approx}\|h_1\cdots h_{r-1}a|_{Y_r}\|_{h_r},\quad\|va\|_{h_r}\overset{N}{\approx}\|va\|_{h_0},
     \end{align}
     \begin{align}\label{eq16}
         \|a|_{Y_r}\|_{h_r}\overset{N}{\approx}\|a|_{Y_r}\|_{h_0},\quad\|a\|_{h_r}\overset{N}{\approx}\|a\|_{h_0}
     \end{align}
     Combining (\ref{eq11})(\ref{eq13})(\ref{eq15}) we get (\ref{isometry3}). Combining (\ref{eq12})(\ref{eq14})(\ref{eq16}) and (\ref{isometry3}) we get (\ref{isometry4}).
 \end{proof}
 \begin{proposition}\label{proposition: hard Lefschetz for mixed ample line bundles}
     Let $X$ be a smooth projective variety of dimension $n$ over $\C$, $p,q,k$ nonnegative integers satisfying $p+q+k=n$, $h_1,\cdots,h_k\in\Amp(X)$. Then the map
    \[-\wedge h_1\cdots h_k:H^{p,q}(X)\rightarrow H^{n-q,n-p}(X)\]
    is an isomorphsm. Moreover, if $Nh_i-h_j\in\Amp(X)\ (\forall1\leq i,j\leq k)$ for some $N\geq 2$, then $\forall a\in H^{p,q}(X)$,
    \begin{align}\label{eq6}
        \|h_1\cdots h_ka\|_{h_1}\overset{N}{\approx}\|a\|_{h_1}.
    \end{align}
 \end{proposition}
\begin{proof}
    It suffices to prove the moreover part. By continuity we may assume that $h_1,\cdots,h_k\in\Amp(X)\cap N^1(X)_\Q$. Let $M$ be a positive integer such that $Mh_1,\cdots,Mh_k$ are all very ample. If we replace $h_1,\cdots,h_k$ by $Mh_1,\cdots,Mh_k$, both sides of (\ref{eq6}) are scaled by the factor $M^{\frac{n-p-q}{2}}$. So we may assume that $h_1,\cdots,h_k$ are all very ample. This case is just (\ref{isometry4}) of lemma \ref{lemma: restriction quasi-isometry 2}.
\end{proof}
\begin{definition}\label{definition: HR global cohomology}
    Let $X$ be a compact K\"ahlerian manifold of dimension $n$, $w\in\Kah(X),\ v\in H^{r,r}(X,\R)\ (0\leq r\leq n)$. $(v,w)$ is called a Hodge-Riemann pair if for all nonnegative integers $p,q,k$ satisfying $p+q+k+r=n$, we have:
    \begin{itemize}
    \item[(1)] (Hard Lefschetz)
    \[-\wedge vw^k:H^{p,q}(X)\rightarrow H^{n-q,n-p}(X)\] is an isomorphism.
    \item[(2)] (Hodge-Riemann relation) The following nondegenerate Hermitian form on $H^{p,q}(X)$
    \[\left<a,b\right>_{vw^k}:=(-1)^q\sqrt{-1}^{(p+q)^2}\int_Xa\bar{b}vw^k\]
    is positive definite on the subspace
    \[P^{p,q}_{(v,w)}(X):=\ker\left(-\wedge vw^{k+1}:H^{p,q}(X)\rightarrow H^{n-q+1,n-p+1}(X)\right).\]
\end{itemize}
\end{definition}
\begin{corollary}
    Let $X$ be a smooth projective variety of dimension $n$, $h_1,\cdots,h_r,h\in\Amp(X)\ (0\leq r\leq n)$. Let $v=h_1\cdots h_r$. Then $(v,h)$ is a Hodge-Riemann pair.
\end{corollary}
\begin{proof}
    By proposition \ref{proposition: hard Lefschetz for mixed ample line bundles} and standard continuity argument.
\end{proof}
Let $(v,w)$ be a Hodge-Riemann pair on $X$. We define an Hermitian metric $\|\cdot\|_{(v,w)}^2$ on $H^{p,q}(X)$ for $p+q\leq n-r$ or $p+q\geq n+r$ as follow. In the following we keep assuming that $p+q\leq n-r$ and define $\|\cdot\|_{(v,w)}^2$ on $H^{p,q}(X)$ and $H^{n-q,n-p}(X)$ separately. By the hard Lefschetz property we have the Lefschetz decomposition:
\[H^{p,q}(X)=\bigoplus_{i=0}^{\min(p,q)}w^iP_{(v,w)}^{p-i,q-i}(X).\]
In other words every $a\in H^{p,q}(X)$ can be written uniquely as
\[a=\sum_{i=0}^{\min(p,q)}w^ia_i\]
with $a_i\in P_{(v,w)}^{p-i,q-i}(X)$. We define
\begin{align}\label{compute the metric by Q3}
    \|a\|_{(v,w)}^2:=\sum_{i=0}^{\min(p,q)}\frac{i!}{(k+i)!}\left<a_i,a_i\right>_{vw^{k+2i}}
\end{align}
where $k=n-r-p-q$.

Every $a\in H^{n-q,n-p}(X)$ can be written uniquely as
\[a=\sum_{i=0}^{\min(p,q)}vw^{k+i}a_i\]
with $a_i\in P_{(v,w)}^{p-i,q-i}(X)$, where $k=n-p-q-r$. We define
\begin{align}\label{compute the metric by Q4}
    \|a\|_{(v,w)}^2:=\sum_{i=0}^{\min(p,q)}\frac{(k+i)!}{i!}\left<a_i,a_i\right>_{vw^{k+2i}}.
\end{align}
\textbf{Rescaling:}
\begin{align}\label{rescaling1}
    \|a\|_{(\lambda v,w)}^2=\lambda^{\pm1}\|a\|_{(v,w)}^2\quad(+\text{ for }p+q\leq n-r,-\text{ for }p+q\geq n+r)
\end{align}
\begin{align}\label{rescaling2}
    \|a\|_{(\lambda^r v,\lambda w)}^2=\lambda^{n-p-q}\|a\|_{(v,w)}^2\quad(\forall p+q\leq n-r\text{ or }p+q\geq n+r)
\end{align}
\begin{remark}\label{remark: continuity}
    It is clear from definition that if $(v_0,w_0)$ is a Hodge-Riemann pair on a compact K\"ahlerian manifold $X$, then every $(v,w)$ in a small neighborhood of $(v_0,w_0)\in H^{1,1}(X,\R)\times H^{r,r}(X,\R)$ is also a Hodge-Riemann pair and the map $(a,v,w)\mapsto\|a\|_{(v,w)}$ is continuous.
\end{remark}
The metric is characterized by the following lemma.
\begin{lemma}\label{lemma: restriction identities}
    Let $X$ be a smooth projective variety of dimension $n$, $h_1,\cdots,h_r,h\in\Amp(X)\ (0\leq r\leq n)$. Let $v=h_1\cdots h_r$. Suppose that $h_1,\cdots,h_r$ are very ample and are represented by smooth divisors $Y_1,\cdots, Y_r$ that intersect transversally. Let $Y=Y_1\cap\cdots\cap Y_r$. Then for $a\in H^{p,q}(X)$, we have
    \begin{align}\label{restriction identity 1}
        \|a\|_{(v,h)}=\|a|_Y\|_h\quad(p+q\leq n-r),
    \end{align}
    \begin{align}\label{restriction identity 2}
        \|va\|_{(v,h)}=\|a|_Y\|_h\quad(p+q\geq n-r).
    \end{align}
\end{lemma}
\begin{proof}
    Since $\int_Ya=\int_Xva$, for $p+q\leq n-r-1$ if we identify $H^{p,q}(X)$ with $H^{p,q}(Y)$ via the natural restriction then $\left<-,-\right>_{vh^k}$ is equal to $\left<-,-\right>_{h^k}$. Note that $P^{p,q}$ is nothing but the orthogonal completement of $hH^{p,q}$ with respect to $\left<-,-\right>$, it follows that the Lefschetz decompositions on these two spaces coincide. The same holds for $p+q=n-r$ except that $H^{p,q}(X)$ is a subspace of $H^{p,q}(Y)$. So comparing (\ref{compute the metric by Q1}) and (\ref{compute the metric by Q3}) we get (\ref{restriction identity 1}).

    If $a\in H^{p,q}(X)$ with $p+q\geq n-r$, $va$ can be written uniquely as $va=\sum_ivh^{k+i}a_i$. Since $\ker\left(-\wedge v:H^{p,q}(X)\rightarrow H^{p+r,q+r}(X)\right)=\ker\left(H^{p,q}(X)\rightarrow H^{p,q}(Y)\right)$, we have $a|_Y=\sum_ih^{k+i}a_i|_Y$. So comparing (\ref{compute the metric by Q2}) and (\ref{compute the metric by Q4}) we get (\ref{restriction identity 2}).
\end{proof}
We end this section by the following lemma, which will be used to prove lemma \ref{lemma: eliminate equivalent local} (a local result). Basically it means that if $v=h_1\cdots h_r$, then dropping some $h_i$ that are equivalent to $h$ does not affect the metric $\|\cdot\|_{(v,h)}$.
\begin{lemma}\label{lemma: eliminate equivalent global}
    Let $X$ be a smooth projective variety of dimension $n$, $0\leq s\leq r\leq n$ nonnegative integers. Let $h_1,\cdots,h_r,h\in\Amp(X)$, $v=h_1\cdots h_r,\ v'=h_1\cdots h_s$. Suppose that $Nh-h_i,Nh_i-h\in\Amp(X)$ for $i=s+1,\cdots,r$. Then
    \begin{align}\label{eq17}
        \|a\|_{(v,h)}\overset{N}{\approx}\|a\|_{(v',h)}
    \end{align}
    for all $a\in H^{p,q}(X)$ with $p+q\leq n-r$ or $p+q\geq n+r$.
\end{lemma}
\begin{proof}
    By continuity (Remark \ref{remark: continuity}) we may assume that $h_1,\cdots,h_r\in\Amp(X)\cap N^1(X)_\Q$. Let $M$ be a positive integer such that $Mh_1,\cdots,Mh_r$ are all very ample. If we replace $h_1,\cdots,h_r,h$ by $Mh_1,\cdots,Mh_r,Mh$, by the rescaling law (\ref{rescaling2}) both sides of the inequality (\ref{eq17}) are scaled by $M^{\frac{n-p-q}{2}}$. The assumption $Nh-h_i,Nh_i-h\in\Amp(X)$ is also invariant. So we may assume that $h_1,\cdots,h_r$ are all very ample. Take transverally intersecting smooth divisors $Y_1,\cdots,Y_r$ representing $h_1,\cdots,h_r$. Let $Y=Y_1\cap\cdots\cap Y_s,\ Z=Y_1\cap\cdots\cap Y_r$.

    First consider the case $p+q\leq n-r$. By lemma \ref{lemma: restriction identities},
    \begin{align*}
        \|a\|_{(v,h)}=\|a|_Z\|_h,\quad\|a\|_{(v',h)}=\|a|_Y\|_h
    \end{align*}
    while applying lemma \ref{lemma: restriction quasi-isometry 2} to $a|_Y$ we get
    \begin{align*}
        \|a|_Y\|_h\overset{N}{\approx}\|a|_Z\|_h.
    \end{align*}
    Combining gives (\ref{eq17}).

    Next consider the case $p+q\geq n+r$. $-\wedge v:H^{p-r,q-r}(X)\rightarrow H^{p,q}(X)$ is surjective (follows from the hard Lefschetz property of a Hodge-Riemann pair), so we can write $a=vb$ for some $b\in H^{p-r,q-r}(X)$. Let $v''=h_{s+1}\cdots h_r$. By lemma \ref{lemma: restriction identities},
    \begin{align*}
        \|vb\|_{(v,h)}=\|b|_Z\|_{h},\quad\|vb\|_{(v',h)}=\|v''b|_Y\|_{h}.
    \end{align*}
    Applying (\ref{isometry3}) of lemma \ref{lemma: restriction quasi-isometry 2} to $b|_Y$ we get
    \begin{align*}
        \|v''b|_Y\|_h\overset{N}{\approx}\|b|_Z\|_h.
    \end{align*}
    Combining gives (\ref{eq17}).
\end{proof}

\section{The harmonic space}

Let $V$ be an $n$-dimensional vector space over $\C$, $\omega\in\bigwedge^{1,1}V$ a positive real (1,1)-form, $\nu\in\bigwedge^{r,r}V$ a real $(r,r)$-form.
\begin{definition}
    The pair $(\nu,\omega)$ is called a Hodge-Riemann pair on $V$ if for all nonnegative integers $p,q,k$ satisfying $p+q+k+r=n$, we have:
\begin{itemize}
    \item[(1)] (Hard Lefschetz)
    \[-\wedge\nu\omega^k:\bigwedge^{p,q}V\rightarrow\bigwedge^{n-q,n-p}V\] is an isomorphism.
    \item[(2)] (Hodge-Riemann relation) The following nondegenerate Hermitian form on $\bigwedge^{p,q}V$
    \[\left<\alpha,\beta\right>_{\nu\omega^k}:=(-1)^q\sqrt{-1}^{(p+q)^2}\int_V\alpha\wedge\bar{\beta}\wedge\nu\omega^k\]
    is positive definite on the subspace
    \[P^{p,q}_{(\nu,\omega)}:=\ker(-\wedge\nu\omega^{k+1}:\bigwedge^{p,q}V\rightarrow\bigwedge^{n-q+1,n-p+1}V).\]
    Here we define $\int_V\alpha=c$ for $\alpha\in\bigwedge^{n,n}V$ if $\alpha=c\frac{\omega^n}{n!}$.
\end{itemize}
\end{definition}

\begin{example}
    Let $\omega,\omega_1,\cdots,\omega_r\in\bigwedge^{1,1}V$ be positive real (1,1)-forms. Let $\nu=\omega_1\cdots\omega_r$. Then $(\nu,\omega)$ is a Hodge-Riemann pair.
\end{example}

Let $(\nu,\omega)$ be a Hodge-Riemann pair on $V$. For every $p+q\leq n-r$ we have a decomposition
\begin{align}\label{Lefschetz decomposition}
    \bigwedge^{p,q}V=\bigoplus_{i=0}^{\min(p,q)}\omega^i P_{(\nu,\omega)}^{p-i,q-i}.
\end{align}
In other words every $\alpha\in\bigwedge^{p,q}V$ can be written uniquely as
\[\alpha=\sum_{i=0}^{\min(p,q)}\omega^i\alpha_i\]
with $\alpha_i\in P^{p-i,q-i}$.

For nonnegative integers $p,q$ satisfying $p+q\leq n-r$ we define $*:\bigwedge^{p,q}V\rightarrow\bigwedge^{n-q,n-p}V$ by
\[*\left(\sum_{i=0}^{\min(p,q)}\omega^i\alpha_i\right):=\sum_{i=0}^{\min(p,q)}\sqrt{-1}^{-(p+q)^2}(-1)^{q-i}\frac{i!}{(k+i)!}\nu\omega^{k+i}\alpha_i,\]
where we denote $k=n-p-q-r$. We then define $*:\bigwedge^{p,q}V\rightarrow\bigwedge^{n-q,n-p}V$ for $p+q\geq n+r$ by requiring that $*^2=(-1)^{p+q}$.

Define
\[(\alpha,\beta):=\int_V\alpha\wedge*\bar{\beta}\]
for $\alpha,\beta\in\bigwedge^{p,q}V$ with $p+q\leq n-r$ or $p+q\geq n+r$. Then by the definition of a Hodge-Riemann pair, $(-,-)$ is a positive-definite Hermitian form on $\bigwedge^{p,q}V$ and makes the decomposition (\ref{Lefschetz decomposition}) an orthogonal decomposition. Moreover, $*:\bigwedge^{p,q}V\rightarrow\bigwedge^{n-q,n-p}V$ is an isometry. The induced Hermitian metric is denoted by $|\cdot|_{(\nu,\omega)}^2$.

\begin{remark}\label{remark: complex torus = linear}
    Let $X$ be a complex torus, then $H^{*,*}(X)\cong \bigwedge^{*,*}T^*_xX$ ($x\in X$ arbitrary) canonically. Since the definition of a Hodge-Riemann pair and the induced metric in either context depends only on the cohomology ring $H^{*,*}(X)$ or the exterior algebra $\bigwedge^{*,*}T_x^*X$, we see that (identifying them via the canonical isomorphism) $(\nu,\omega)$ is a Hodge-Riemann pair on $X$ if and only if it is a Hodge-Riemann pair on $T^*_xX$. If furthermore $\omega$ satisfies $\int_X\omega^n/n!=1$, then the metrics $\|\cdot\|_{(\nu,\omega)}$ and $|\cdot|_{(\nu,\omega)}$ coincide.
\end{remark}

Denote by $L:=-\wedge\omega:\bigwedge^{p,q}V\rightarrow\bigwedge^{p+1,q+1}V$. For $p+q\leq n-r$ define
\[\Lambda:=(-1)^{p+q}*L*:\bigwedge^{p,q}V\rightarrow\bigwedge^{p-1,q-1}V.\]
By direct computation one sees that
\[\Lambda\left(\sum_{i=0}^{\min(p,q)}\omega^i\alpha_i\right)=\sum_{i=1}^{\min(p,q)}i(k+i+1)\omega^{i-1}\alpha_i,\]
where $k=n-p-q-r$. So that $[\Lambda,L]=n-r-p-q$ on $\bigwedge^{p,q}V\quad (p+q\leq n-r-2)$, and $\alpha$ is primitive $\Leftrightarrow$ $\Lambda\alpha=0$.

Now let $X$ be a compact K\"ahlerian manifold of dimension $n$, $\omega$ a K\"ahler form on $X$, $\nu\in\mathcal{A}^{r,r}(X)$ a $d$-closed, real $(r,r)$-form. The pair $(\nu,\omega)$ is called a Hodge-Riemann pair if for every $x\in X$, $(\nu_x,\omega_x)$ is a Hodge-Riemann pair on $T^*_xX$.

Let $(\nu,\omega)$ be a Hodge-Riemann pair on $X$. The operators $L,\Lambda,*$ globalize in the obvious way. The inner product $(-,-)$ on $\mathcal{A}^{p,q}(X)$ is defined by
\begin{align}\label{global inner product}
    (\alpha,\beta):=\int_X(\alpha,\beta)\frac{\omega^n}{n!}=\int_X\alpha\wedge*\bar{\beta}.
\end{align}
We further define $\bar{\partial}^*:=-*\partial*,\ \partial^*:=-*\bar{\partial}*,\ d^*:=\partial^*+\bar{\partial}^*$ on $\mathcal{A}^{p,q}(X)$ for $p+q\leq n-r$ or $p+q\geq n+r+1$. By the stokes formula one obtain $(\partial^*\alpha,\beta)=(\alpha,\partial\beta)$ and $(\bar{\partial}^*\alpha,\gamma)=(\alpha,\bar{\partial}\gamma)$ for $\alpha\in\mathcal{A}^{p,q}(X),\ \beta\in\mathcal{A}^{p-1,q}(X),\ \gamma\in\mathcal{A}^{p,q-1}(X)$ with $p+q\leq n-r$ or $p+q\geq n+r+1$.
\begin{lemma}\label{basic identity}
    \begin{itemize}
        \item[(1)] If $p+q\leq n-r-1$, then
        \[\bar{\partial}^*=\sqrt{-1}[\partial,\Lambda],\ \partial^*=-\sqrt{-1}[\bar{\partial},\Lambda]\]
        on $\mathcal{A}^{p,q}(X)$.
        \item[(2)] If $p+q=n-r$, $\alpha\in\mathcal{A}^{p,q}(X)$,
        \[\begin{cases}
            \text{if }\partial\alpha=0,\text{ then }\bar{\partial}^*\alpha=\sqrt{-1}\partial\Lambda\alpha;\\
            \text{if }\bar{\partial}\alpha=0,\text{ then }\partial^*\alpha=-\sqrt{-1}\bar{\partial}\Lambda\alpha.
        \end{cases}\]
    \end{itemize}
\end{lemma}
\begin{proof}
We only prove the formulas $\partial^*=-\sqrt{-1}[\bar{\partial},\Lambda]$ and $\bar{\partial}^*\alpha=\sqrt{-1}\partial\Lambda\alpha$. The proof of the other two are similar. We first prove that $\partial^*(\omega^i\alpha)=-\sqrt{-1}[\bar{\partial},\Lambda](\omega^i\alpha)$, where $\alpha\in C^{\infty}(X,P^{p-i,q-i}),\ p+q\leq n-r-1$ ($P^{p,q}$ are now subbundles of $\bigwedge^{p,q}T^*X$). This is equivalent to
\begin{align}\label{eq1}
    \bar{\partial}*(\omega^i\alpha)=(-1)^{p+q+1}\sqrt{-1}*[\bar{\partial},\Lambda](\omega^i\alpha).
\end{align}
By assumption $\alpha\in C^{\infty}(X,P^{p-i,q-i})$, so $\nu\omega^{k+i+1}\alpha=0$, where $k=n-p-q-r$. We get $\nu\omega^{k+i+1}\bar{\partial}\alpha=0$. So we can write $\bar{\partial}\alpha=\alpha_0+\omega\alpha_1$, with $\alpha_0\in C^{\infty}(X,P^{p-i,q-i+1}),\ \alpha_1\in C^{\infty}(X,P^{p-i-1,q-i})$.
\begin{align*}
    \begin{split}
        &\bar{\partial}\Lambda(\omega^i\alpha)=\bar{\partial}(i(k+i+1)\omega^{i-1}\alpha)=i(k+i+1)\omega^{i-1}(\alpha_0+\omega\alpha_1),\\
        &\Lambda\bar{\partial}(\omega^i\alpha)=\Lambda(\omega^i\alpha_0+\omega^{i+1}\alpha_1)=i(k+i)\omega^{i-1}\alpha_0+(i+1)(k+i+1)\omega^i\alpha_1.\\
        &\Rightarrow[\bar{\partial},\Lambda](\omega^i\alpha)=i\omega^{i-1}\alpha_0-(k+i+1)\omega^i\alpha_1.\\
    \end{split}
\end{align*}
\begin{align*}
    \begin{split}
        &(-1)^{p+q+1}\sqrt{-1}*[\bar{\partial},\Lambda](\omega^i\alpha)\\
        =&(-1)^{p+q+1}\sqrt{-1}*(i\omega^{i-1}\alpha_0-(k+i+1)\alpha_1)\\
        =&(-1)^{p+q+1}\sqrt{-1}\bigg[i\sqrt{-1}^{-(p+q-1)^2}(-1)^{q-i+1}\frac{(i-1)!}{(k+i)!}\nu\omega^{k+i}\alpha_0\\
        &-(k+i+1)\sqrt{-1}^{-(p+q-1)^2}(-1)^{q-i}\frac{i!}{(k+i+1)!}\nu\omega^{k+i+1}\alpha_1\bigg]\\
        =&\sqrt{-1}^{-(p+q)^2}(-1)^{q-i}\left[\frac{i!}{(k+i)!}\nu\omega^{k+i}\alpha_0+\frac{i!}{(k+i)!}\nu\omega^{k+i+1}\alpha_1\right]\\
        =&\bar{\partial}*(\omega^i\alpha).
    \end{split}
\end{align*}
This proves (\ref{eq1}), thus the formula $\partial^*=-\sqrt{-1}[\bar{\partial},\Lambda]$ on $\mathcal{A}^{p,q}(X)$ with $p+q\leq n-r-1$.

Next we prove that if $p+q=n-r$, $\alpha=\sum_{i=0}^{\min(p,q)}\omega^i\alpha_i\in\mathcal{A}^{p,q}(X)$ satisfies $\partial\alpha=0$, then $\bar{\partial}^*\alpha=\sqrt{-1}\partial\Lambda\alpha$.

As before we can write $\partial\alpha_i=\alpha_i'+\omega\alpha_i''$ for $i\geq 1$, where $\alpha_i'\in C^{\infty}(X,P^{p-i+1,q-i}),\ \alpha_i''\in C^{\infty}(X,P^{p-i,q-i-1})$. Then the assumption $\partial\alpha=0$ implies
\begin{align}\label{eq2}
    0=\partial\alpha=\partial\alpha_0+\omega\alpha_1'+\sum_{i\geq 1}\omega^{i+1}(\alpha_i''+\alpha_{i+1}').
\end{align}
We have $\alpha_i''+\alpha_{i+1}'\in C^{\infty}(X,P^{p-i,q-i-1})$. Multiplying (\ref{eq2}) by $\nu\omega^l\ (l\geq 1)$ we get
\begin{align}\label{eq3}
    \nu\omega^{2l+1}\left[\sum_{i\geq l}\omega^{i-l}(\alpha_i''+\alpha_{i+1}')\right]=0.
\end{align}
Using the fact that $\nu\omega^{2l+1}:\bigwedge^{p-l,q-l-1}T^*X\rightarrow\bigwedge^{n-q+l+1,n-p+l}T^*X$ is an isomorphism, (\ref{eq3}) implies
\[\sum_{i\geq l}\omega^{i-l}(\alpha_i''+\alpha_{i+1}')=0\quad(\forall l\geq1).\]
We get
\begin{align}\label{eq4}
    \alpha_i''+\alpha_{i+1}'=0\quad (\forall i\geq 1)\quad\text{and}\quad\partial\alpha_0+\omega\alpha_1'=0.
\end{align}
Now we compute
\begin{align*}
    \begin{split}
        &\partial*\alpha\\
        =&\partial\left(\sum_{i\geq 0}(-1)^{q-i}\sqrt{-1}^{-(p+q)^2}\nu\omega^i\alpha_i\right)\\
        =&(-1)^q\sqrt{-1}^{-(p+q)^2}\left[\nu\partial\alpha_0+\sum_{i\geq 1}(-1)^i(\nu\omega^i\alpha_i'+\nu\omega^{i+1}\alpha_i'')\right],
    \end{split}
\end{align*}
\begin{align*}
    \begin{split}
        &\sqrt{-1}*\partial\Lambda\alpha\\
        =&\sqrt{-1}*\partial\left[\sum_{i\geq 1}i(i+1)\omega^{i-1}\alpha_i\right]\\
        =&\sqrt{-1}*\left[\sum_{i\geq 1}\left(i(i+1)\omega^{i-1}\alpha_i'+i(i+1)\omega^i\alpha_i''\right)\right]\\
        =&\sqrt{-1}\sqrt{-1}^{-(p+q-1)^2}(-1)^q\left[\sum_{i\geq 1}\left((-1)^i(i+1)\nu\omega^i\alpha_i'+(-1)^{i+1}i\nu\omega^{i+1}\alpha_i''\right)\right]\\
        =&\sqrt{-1}^{-(p+q)^2}(-1)^p\left[\sum_{i\geq 1}(-1)^i(\nu\omega^i\alpha_i'+\nu\omega^{i+1}\alpha_i'')-\nu\omega\alpha_1'+\sum_{i\geq 1}(-1)^{i+1}(i+1)\nu\omega^{i+1}(\alpha_i''+\alpha_{i+1}')\right]\\
        =&\sqrt{-1}^{-(p+q)^2}(-1)^p\left[\sum_{i\geq 1}(-1)^i(\nu\omega^i\alpha_i'+\nu\omega^{i+1}\alpha_i'')+\nu\partial\alpha_0\right],
    \end{split}
\end{align*}
where in the last step we have used equation (\ref{eq4}).

Comparing the two equations above we see that $\partial*\alpha=(-1)^{p+q}\sqrt{-1}*\partial\Lambda\alpha$. Equivalently, $\bar{\partial}^*\alpha=\sqrt{-1}\partial\Lambda\alpha$, which is what we want.
\end{proof}
Define the operators $\Delta_\partial:=\partial\partial^*+\partial^*\partial$, $\Delta_{\bar{\partial}}:=\bar{\partial}\bar{\partial}^*+\bar{\partial}^*\bar{\partial}$, $\Delta_d=dd^*+d^*d$ on $\mathcal{A}^{p,q}(X)$ for $p+q\leq n-r-1$ or $p+q\geq n+r+1$. The following lemma  follows from the previous lemma by simple computation.
\begin{lemma}\label{kahler identies}
    \[\Delta_\partial=\Delta_{\bar{\partial}}=\frac{1}{2}\Delta_d,\ [\Delta_d,*]=0,\ \partial^*\bar{\partial}+\bar{\partial}\partial^*=0,\ \bar{\partial}^*\partial+\partial\bar{\partial}^*=0\]
    on $\mathcal{A}^{p,q}(X)$ with $p+q\leq n-r-1$ or $p+q\geq n+r+1$;
    \[[\Lambda,\Delta_d]=0\]
    on $\mathcal{A}^{p,q}(X)$ with $p+q\leq n-r-1$;
    \[[L,\Delta_d]=0\]
    on $\mathcal{A}^{p,q}(X)$ with $p+q\leq n-r-3$ or $p+q\geq n+r+1$.
\end{lemma}
\begin{proof}
    The formulas $\Delta_\partial=\Delta_{\bar{\partial}}=\frac{1}{2}\Delta_d,\ \partial^*\bar{\partial}+\bar{\partial}\partial^*=0,\ \bar{\partial}^*\partial+\partial\bar{\partial}^*=0$ are obtained easily by computation using lemma \ref{basic identity} in the range $p+q\leq n-r-1$. Then one use the obvious identity $*\Delta_{\bar{\partial}}=\Delta_{\partial}*$ to extend the formulas to the range $p+q\geq n+r+1$. As for $[\Lambda,\Delta_d]=0$, we compute
    \[[\Lambda,\Delta_\partial]=[\Lambda,\partial]\partial^*+\partial^*[\Lambda,\partial]=-\sqrt{-1}(\bar{\partial}^*\partial^*+\partial^*\bar{\partial}^*)=\pm\sqrt{-1}*(\bar{\partial}\partial+\partial\bar{\partial})*=0.\]
    Taking $*$ we get the formula $[L,\Delta_d]=0$ in the range $p+q\geq n+r+1$. Finally, the formula $[L,\Delta_d]=0$ in the range $p+q\leq n-r-3$ follows easily from the following identities:
    \begin{align}\label{eq5}
        \begin{split}
            &[L,\partial^*]=-\sqrt{-1}\bar{\partial}\text{ on }\mathcal{A}^{p,q}(X)\text{ with }p+q\leq n-r-3;\\
            &[L,\partial^*]\alpha=0\text{ if }\alpha\in\mathcal{A}^{p,q}(X)\text{ with }p+q=n-r-2\text{ satisfies }\bar{\partial}\alpha=0.
        \end{split}
    \end{align}
    To prove (\ref{eq5}), we use lemma \ref{basic identity}. If $p+q\leq n-r-3$,
    \[[L,\partial^*]=-\sqrt{-1}[L,[\Lambda,\bar{\partial}]]=\sqrt{-1}[\bar{\partial},[L,\Lambda]]=-\sqrt{-1}\bar{\partial}.\]
    If $p+q=n-r-2$, $\alpha\in\mathcal{A}^{p,q}(X)$ satisfies $\bar{\partial}\alpha=0$, then $\bar{\partial}L\alpha=0$,
    \[[L,\partial^*]\alpha=L\partial^*\alpha-\partial^*L\alpha=\sqrt{-1}L[\Lambda,\bar{\partial}]\alpha+\sqrt{-1}\bar{\partial}\Lambda L\alpha=0.\]
\end{proof}
\begin{definition}
    Let $(\nu,\omega)$ be a Hodge-Riemann pair on $X$. For $p+q\leq n-r$ or $p+q\geq n+r$, define
    \[\mathcal{H}^{p,q}(X):=\left\{\alpha\in\mathcal{A}^{p,q}(X)\Big|d\alpha=0,\ (\alpha,\bar{\partial}\beta)=0,\ (\alpha,\partial\gamma)=0,\ \forall\beta\in\mathcal{A}^{p,q-1}(X),\ \gamma\in\mathcal{A}^{p-1,q}(X)\right\},\]
    where $(-,-)$ is the inner product induced by $(\nu,\omega)$ (see (\ref{global inner product})).
\end{definition}
$\mathcal{H}^{p,q}(X)$ is called the space of harmonic $(p,q)$-forms (with respect to $(\nu,\omega)$). It is the space of $d$-closed $(p,q)$-forms that are orthogonal to $\bar{\partial}\mathcal{A}^{p,q-1}(X)+\partial\mathcal{A}^{p-1,q}(X)$ (with respect to the inner product $(-,-)$ induced by $(\nu,\omega)$). (Note that in order to define the harmonic space in bidegree $(p,q)$, one only need to specify an inner product on $\bigwedge^{p,q}T^*X$.)
\begin{lemma}\label{Hodge star characterization of harmonic space}
    \[\mathcal{H}^{p,q}(X)=\left\{\alpha\in\mathcal{A}^{p,q}(X)\big|d\alpha=0,\ d*\alpha=0\right\}.\]
    ($p+q\leq n-r$ or $p+q\geq n+r$)
\end{lemma}
\begin{proof}
    Use (\ref{global inner product}) and integration by part.
\end{proof}
\begin{corollary}\label{Hodge star preserves harmonic spaces}
    \[*\mathcal{H}^{p,q}(X)=\mathcal{H}^{n-q,n-p}(X).\]
    ($p+q\leq n-r$ or $p+q\geq n+r$)
\end{corollary}
\begin{lemma}\label{Laplacian characterization of harmonic space}
    If $p+q\leq n-r-1$ or $p+q\geq n+r+1$, then
    \begin{align*}
        \begin{split}
            \mathcal{H}^{p,q}(X)=\ker\Delta_d=&\left\{\alpha\in\mathcal{A}^{p,q}(X)\big|\partial\alpha=0,\ \partial^*\alpha=0\right\}\\
            =&\left\{\alpha\in\mathcal{A}^{p,q}(X)\big|\bar{\partial}\alpha=0,\ \bar{\partial}^*\alpha=0\right\}\\
            =&\left\{\alpha\in\mathcal{A}^{p,q}(X)\big|d\alpha=0,\ d^*\alpha=0\right\}.
        \end{split}
    \end{align*}
\end{lemma}
\begin{proof}
    $\ker\Delta_d$ is clear the same as the last three spaces. (For example, $\Delta_d\alpha=0\Rightarrow(\partial\partial^*\alpha+\partial^*\partial\alpha,\alpha)=0\Rightarrow(\partial^*\alpha,\partial^*\alpha)=(\partial\alpha,\partial\alpha)=0\Rightarrow\partial\alpha=0,\ \partial^*\alpha=0$.) The last space is the same as $\mathcal{H}^{p,q}(X)$ by lemma \ref{Hodge star characterization of harmonic space}.
\end{proof}
$\Delta_{\bar{\partial}}$ and $\Delta_{\partial}$ are second order elliptic self-adjoint (with respect to (-,-)) operators on $\mathcal{A}^{p,q}(X)$ ($p+q\leq n-r-1$ or $p+q\geq n+r+1$). By lemma \ref{kahler identies}, so is $\Delta_d$. By standard theory on elliptic differential operators, we have $\dim\mathcal{H}^{p,q}(X)<\infty$ and
\begin{align}\label{harmonic decomposition}
    \mathcal{A}^{p,q}(X)=\mathcal{H}^{p,q}(X)\oplus\Delta_d\mathcal{A}^{p,q}(X)
\end{align}
(orthogonal direct sum), where $p+q\leq n-r-1$ or $p+q\geq n+r+1$.
\begin{lemma}\label{harmonic space at the dangerous degree}
    If $p+q=n-r$, then
    \begin{align*}
        \begin{split}
            \mathcal{H}^{p,q}(X)=&\left\{\alpha\in\mathcal{A}^{p,q}(X)\big|d\alpha=0,\ \partial\bar{\partial}\Lambda\alpha=0\right\}\\
            =&\left\{\alpha\in\mathcal{A}^{p,q}(X)\big|d\alpha=0,\ \partial^*\bar{\partial}^*\alpha=0\right\}
        \end{split}
    \end{align*}
    If we denote the space of $d$-closed $(p,q)$-forms by $\mathcal{Z}^{p,q}(X)$, then
    \[\mathcal{Z}^{p,q}(X)=\mathcal{H}^{p,q}(X)\oplus\partial\bar{\partial}\mathcal{A}^{p-1,q-1}(X)\]
    (orthogonal direct sum).
\end{lemma}
\begin{proof}
    By lemma \ref{Hodge star characterization of harmonic space}, $\mathcal{H}^{p,q}(X)=\left\{\alpha\in\mathcal{A}^{p,q}(X)\big|d\alpha=0,\ d^*\alpha=0\right\}$. Using lemma \ref{basic identity}, this space is equal to $\left\{\alpha\in\mathcal{A}^{p,q}(X)\big|d\alpha=0,\ \partial\Lambda\alpha=0,\ \bar{\partial}\Lambda\alpha=0\right\}$, which is clearly contained in $\left\{\alpha\in\mathcal{A}^{p,q}(X)\big|d\alpha=0,\ \partial\bar{\partial}\Lambda\alpha=0\right\}$. To prove the converse inclusion, suppose $d\alpha=0,\ \partial\bar{\partial}\Lambda\alpha=0$. By lemma \ref{basic identity}, $\partial\partial^*\alpha=-\sqrt{-1}\partial\bar{\partial}\Lambda\alpha=0$. So $(\partial^*\alpha,\partial^*\alpha)=(\alpha,\partial\partial^*\alpha)=0$, which implies $\partial^*\alpha=0$. Similarly we have $\bar{\partial}^*\alpha=0$. This proves $\left\{\alpha\in\mathcal{A}^{p,q}(X)\big|d\alpha=0,\ d^*\alpha=0\right\}\subset\left\{\alpha\in\mathcal{A}^{p,q}(X)\big|d\alpha=0,\ \partial\bar{\partial}\Lambda\alpha=0\right\}$, and so $\mathcal{H}^{p,q}(X)=\left\{\alpha\in\mathcal{A}^{p,q}(X)\big|d\alpha=0,\ \partial\bar{\partial}\Lambda\alpha=0\right\}$.

    Next, by lemma \ref{basic identity}, $\left\{\alpha\in\mathcal{A}^{p,q}(X)\big|d\alpha=0,\ \partial^*\bar{\partial}^*\alpha=0\right\}=\left\{\alpha\in\mathcal{A}^{p,q}(X)\big|d\alpha=0,\ \partial^*\partial\Lambda\alpha=0\right\}=\left\{\alpha\in\mathcal{A}^{p,q}(X)\big|d\alpha=0,\ \partial\Lambda\alpha=0\right\}$, which is equal to $\mathcal{H}^{p,q}(X)$ by the first part of the proof.

    Finally we prove the orthogonal direct sum. Clearly the two spaces $\mathcal{H}^{p,q}(X)$ and $\partial\bar{\partial}\mathcal{A}^{p-1,q-1}(X)$ are orthogonal to each other. So it suffices to prove that their sum is equal to $\mathcal{Z}^{p,q}(X)$. Pick any $\alpha\in\mathcal{Z}^{p,q}(X)$. Since $\partial^*\bar{\partial}^*\alpha$ is orthogonal to $\mathcal{H}^{p-1,q-1}(X)$, using $\mathcal{A}^{p-1,q-1}(X)=\mathcal{H}^{p-1,q-1}(X)\oplus\Delta_d\mathcal{A}^{p-1,q-1}(X)$, we can write $\partial^*\bar{\partial}^*\alpha=\Delta_{\bar{\partial}}\Delta_{\partial}\beta$. I claim that $\Delta_{\bar{\partial}}\Delta_{\partial}\beta=-\partial^*\bar{\partial}^*\partial\bar{\partial}\beta$.
    
    Indeed, $\bar{\partial}^*\Delta_{\bar{\partial}}\Delta_{\partial}\beta=\bar{\partial}^*\partial^*\bar{\partial}^*\alpha=0$ implies $\Delta_{\bar{\partial}}\Delta_{\partial}\beta=\bar{\partial}^*\bar{\partial}\Delta_{\partial}\beta$, which is equal to $\Delta_{\partial}\bar{\partial}^*\bar{\partial}\beta$ since both $\bar{\partial}$ and $\bar{\partial}^*$ commute with $\Delta_d$ (lemma \ref{kahler identies}). Now $\partial^*\Delta_{\partial}\bar{\partial}^*\bar{\partial}\beta\beta=\partial^*\partial^*\bar{\partial}^*\alpha=0$ implies $\Delta_{\partial}\bar{\partial}^*\bar{\partial}\beta=\partial^*\partial\bar{\partial}^*\bar{\partial}\beta$, which is equal to $-\partial^*\bar{\partial}^*\partial\bar{\partial}\beta$ by the anti-commutative of $\partial$ and $\bar{\partial}^*$ (lemma \ref{kahler identies}). This proves the claim.

    Therefore $\partial^*\bar{\partial}^*\left(\alpha+\partial\bar{\partial}\beta\right)=0$. This implies $\alpha+\partial\bar{\partial}\beta\in\mathcal{H}^{p,q}(X)$ by the first part of the lemma. Therefore $\mathcal{Z}^{p,q}(X)=\mathcal{H}^{p,q}(X)+\partial\bar{\partial}\mathcal{A}^{p-1,q-1}(X)$, finishing the proof.
\end{proof}
Recall that $H^{p,q}(X)=\frac{\left\{d\text{-closed }(p,q)\text{-forms}\right\}}{\left\{d\text{-exact }(p,q)\text{-forms}\right\}}$. The natural map $\mathcal{H}^{p,q}(X)\rightarrow H^{p,q}(X)$ is clearly injective ($p+q\leq n-r$ or $p+q\geq n+r$).
\begin{proposition}
    \[\mathcal{H}^{p,q}(X)\cong H^{p,q}(X).\]
    ($p+q\leq n-r$ or $p+q\geq n+r$)
\end{proposition}
\begin{proof}
    It suffices to prove that the natural map is surjective, or that the two spaces have the same dimension. By corollary \ref{Hodge star preserves harmonic spaces}, $\dim_{\C}\mathcal{H}^{p,q}(X)=\dim_{\C}\mathcal{H}^{n-q,n-p}(X)$. So it suffices to prove the surjectivity of $\mathcal{H}^{p,q}(X)\rightarrow H^{p,q}(X)$ in the range $p+q\leq n-r$, which follows from (\ref{harmonic decomposition}) and lemma \ref{harmonic space at the dangerous degree}.
\end{proof}
\begin{lemma}
    \begin{align*}
        \begin{split}
            &\Lambda\left(\mathcal{H}^{p,q}(X)\right)\subset\mathcal{H}^{p-1,q-1}(X)\quad(p+q\leq n-r);\\
            &L\left(\mathcal{H}^{p-1,q-1}(X)\right)\subset\mathcal{H}^{p,q}(X)\quad(p+q\leq n-r\text{ or }p+q\geq n-r+2).
        \end{split}
    \end{align*}
\end{lemma}
\begin{proof}
    The two inclusions in the range $p+q\leq n-r-1$ follow from lemma \ref{kahler identies}, $[L,\Delta_d]=0,\ [\Lambda,\Delta_d]=0$ and lemma \ref{Laplacian characterization of harmonic space}, $\mathcal{H}^{p,q}(X)=\ker\Delta_d$.
    
    Now let $p+q=n-r$. Suppose $\alpha\in\mathcal{H}^{p,q}(X)$. Then $\partial\Lambda\alpha=-\sqrt{-1}\bar{\partial}^*\alpha=0$, $\partial^*\Lambda\alpha=\Lambda\partial^*\alpha=0$. So $\Lambda\alpha\in\mathcal{H}^{p-1,q-1}(X)$. Next suppose $\alpha\in\mathcal{H}^{p-1,q-1}(X)$. Then $dL\alpha=0$, $\partial^*L\alpha=[\partial^*,L]\alpha=0$ by (\ref{eq5}). So $L\alpha\in\mathcal{H}^{p,q}(X)$.
    
    Finally the inclusion $L\left(\mathcal{H}^{p-1,q-1}(X)\right)\subset\mathcal{H}^{p,q}(X)$ in the range $p+q\geq n-r+2$ follows from the inclusion $\Lambda\left(\mathcal{H}^{p,q}(X)\right)\subset\mathcal{H}^{p-1,q-1}(X)$ in the range $p+q\leq n-r$ by taking $*$.
\end{proof}
Define $\mathcal{P}^{p,q}(X):=\ker\left\{\Lambda:\mathcal{H}^{p,q}(X)\rightarrow\mathcal{H}^{p-1,q-1}(X)\right\}$ for $p+q\leq n-r$.
\begin{proposition}\label{proposition: Lefschetz decomposition}
    \[\mathcal{H}^{p,q}(X)=\bigoplus_{i\geq 0}L^i\mathcal{P}^{p-i,q-i}(X)\]
    (orthogonal direct sum), where $p+q\leq n-r$.
\end{proposition}
\begin{proof}
    The natural map $\bigoplus_{i\geq 0}L^i\mathcal{P}^{p-i,q-i}(X)\rightarrow\mathcal{H}^{p,q}(X)$ is clearly injective. Use induction on $p+q$ to prove surjectivity. Indeed, $\forall\alpha\in\mathcal{H}^{p,q}(X)$, if $\Lambda\alpha\in\bigoplus_{i\geq 0}L^i\mathcal{P}^{p-1-i,q-1-i}(X)$, then $\exists\beta\in\bigoplus_{i\geq 1}L^i\mathcal{P}^{p-i,q-i}(X)$ such that $\Lambda\beta=\Lambda\alpha$. So $\alpha-\beta\in\mathcal{P}^{p,q}(X)$.
\end{proof}
\begin{corollary}
    \[-\wedge\nu\omega^k:\mathcal{H}^{p,q}(X)\xrightarrow{\cong}\mathcal{H}^{n-q,n-p}(X),\]
    where $p+q+k=n-r$.
\end{corollary}
\begin{proof}
    \[-\wedge\nu\omega^k:L^i\mathcal{P}^{p-i,q-i}(X)\xrightarrow{\cong}*L^i\mathcal{P}^{p-i,q-i}(X),\]
    then use proposition \ref{Lefschetz decomposition} and corollary \ref{Hodge star preserves harmonic spaces}.
\end{proof}
The Hermitian form $\left<-,-\right>_{\nu\omega^k}$ is positive definite on $\mathcal{P}^{p,q}(X)$ (since it is pointwise positive definite on the bundle $P^{p,q}$). We collect the results in the following theorem.
\begin{theorem}\label{theorem: harmonic space}
    Let $(\nu,\omega)$ be a Hodge-Riemann pair on $X$. Then $([\nu],[\omega])$ is a Hodge-Riemann pair on $X$ in the sense of definition \ref{definition: HR global cohomology}. We have
    \begin{align}\label{global norm = local norm}
        \|a\|_{([\nu],[\omega])}^2=\inf_{[\alpha]=a}\int_X|\alpha|_{(\nu,\omega)}^2\frac{\omega^n}{n!}
    \end{align}
    where $a\in H^{p,q}(X)$ with $p+q\leq n-r$ or $p+q\geq n-r$. Moreover, we can define the harmonic spaces $\H^{p,q}_{(\nu,\omega)}(X)\subset\A^{p,q}(X)$ with $\H^{p,q}_{(\nu,\omega)}(X)\cong H^{p,q}(X)$ such that every $\alpha\in\H^{p,q}_{(\nu,\omega)}(X)$ is the unique minimizer of RHS of (\ref{global norm = local norm}).
\end{theorem}

\begin{corollary}(Dinh-Nguyen\cite{DN2})
    Let $\omega,\omega_1,\cdots,\omega_r$ be K\"ahler forms on $X$, then for $p+q=n-r$,
    \begin{itemize}
        \item[(1)] (Hard Lefschetz)
        \[-\wedge\omega_1\cdots\omega_r:H^{p,q}(X)\xrightarrow{\cong}H^{n-q,n-p}(X).\]
        \item[(2)] (Hodge-Riemann relation) the Hermitian form
        \[\left<\alpha,\beta\right>:=(-1)^q\sqrt{-1}^{(p+q)^2}\int_X\alpha\wedge\bar{\beta}\wedge\omega_1\cdots\omega_r\]
        is positive definite on
        \[P^{p,q}(X):=\ker\left((-\wedge\omega_1\cdots\omega_r\omega:H^{p,q}(X)\rightarrow H^{n-q+1,n-p+1}(X)\right).\]
    \end{itemize}
\end{corollary}

\begin{corollary}\label{corollary: pointwise zero}
    Let $\omega_1,\cdots,\omega_r$ be K\"ahler forms on a compact K\"ahlerian manifold $X$. Suppose that $a\in H^{p,q}(X)$ satisfies $a\cdot[\omega_1\cdots\omega_r]=0$, then there exists a $d$-closed $\alpha\in\A^{p,q}(X)$ repesenting $a$ such that $\alpha\cdot\omega_1\cdots\omega_r=0$. 
\end{corollary}
\begin{proof}
    Use induction on $r$. If $r\leq \max\{0,n-p-q\}$, $-\wedge\omega_1\cdots\omega_r: H^{p,q}(X)\rightarrow H^{p+r,q+r}$ is injective, so we must have $a=0$ and the assertion is trivial. Now suppose $r\geq \max\{1,n-p-q+1\}$. Let $k=p+q+r-n-1$. Let $b=a\cdot[\omega_1\cdots\omega_{r-1}]\in H^{p+r-1,q+r-1}(X)$. Consider the Hodge-Riemann pair $(\omega_1\cdots\omega_{r-1},\omega_r)$. $-\wedge\omega_1\cdots\omega_{r-1}\omega_r^k: H^{p-k,q-k}(X)\rightarrow H^{p+r-1,q+r-1}(X)$ is an isomorphism. Take $\beta\in \H^{p-k,q-k}(X)$ such that $[\beta\cdot \omega_1\cdots\omega_{r-1}\omega_r^k]=b$. By assumption $b\cdot[\omega_r]=a\cdot[\omega_1\cdots\omega_r]=0$. So $\beta\in\mathcal{P}^{p-k,q-k}_{(\omega_1\cdots\omega_{r-1},\omega_r)}(X)$, which implies
    \begin{align}\label{eq22}
        \beta\cdot\omega_1\cdots\omega_{r-1}\omega_r^{k+1}=0.
    \end{align}
    Note that $[\beta\omega_1\cdots\omega_{r-1}\omega_r^k]=b=a\cdot[\omega_1\cdots\omega_{r-1}]$ implies $(a-[\beta\omega_r^k])\cdot[\omega_1\cdots\omega_{r-1}]=0$. By induction hypothesis, $a-[\beta\omega_r^k]$ is represented by some $\gamma\in\A^{p,q}(X)$ such that $\gamma\cdot\omega_1\cdots\omega_{r-1}=0$. Combining with (\ref{eq22}), we see that $\alpha=\gamma+\beta\omega_r^k$ satisfies the desired property.
\end{proof}

\section{Some local results}

\subsection{Outline}
This section is devoted to the proof of the following two linear algebraic propositions.
\begin{proposition}\label{proposition: local 1}
    Let $V$ be an $n$ dimensional $\C$-vector space, $\omega_1,\cdots,\omega_r,\omega$ positive (1,1)-forms on $V$. Let $\nu=\omega_1\cdots\omega_r$. Suppose that $\omega_i\leq N\omega\ (i=1,\cdots,r)$. Then for all $\alpha\in\bigwedge^{p,q}V$ with $p+q\leq n-r$, we have
    \begin{align*}
        |\nu\alpha|_\omega\lesssim N^{r/2}|\alpha|_{(\nu,\omega)}
    \end{align*}
\end{proposition}
\begin{proposition}\label{proposition: local 2}
    Let $V,W$ be $\C$-vector spaces of dimension $n,m$ respectively and $\pi:V\rightarrow W$ a linear map. Let $\omega_1,\cdots,\omega_r$ be positive (1,1)-forms on $W$ and $\omega$ a positive (1,1)-form on $V$. Suppose that $\pi^*\omega_i\leq N\omega$ for $i=1,\cdots,r$. Let $\nu=\pi^*\omega_1\cdots\pi^*\omega_r$. For $\epsilon>0$, define $\nu_{\epsilon}=(\pi^*\omega_1+\epsilon\omega)\cdots(\pi^*\omega_s+\epsilon\omega)$. Let $\alpha\in\bigwedge^{p,q}V$ with $p+q\geq n-r$. Suppose that $\nu\alpha=0$, then
    \begin{align*}
        |\nu_\epsilon\alpha|_{(\nu_\epsilon,\omega)}^2\overset{N}{\lesssim}\epsilon|\alpha|_\omega^2.
    \end{align*}
\end{proposition}
The proofs of these two simple looking propositions are in fact quite technical. They are both proven by the corresponding geometric statements (proposition \ref{proposition: local geometric 1} and corollary \ref{corollary: perturb nef classes 2} below). The two linear algebraic propositions above follow from the corresponding geometric statements applied to abelian varieties (thanks to remark \ref{remark: complex torus = linear}).

The proofs of the geometric statements (proposition \ref{proposition: local geometric 1} and corollary \ref{corollary: perturb nef classes 2}) are quite involved, especially corollary \ref{corollary: perturb nef classes 2}. Some nontrivial techniques from analysis are even used (e.g. lemma \ref{lemma: log cut-off}). It would be interesting to see whether elementary proofs of these linear algebraic facts could be found.

As a warm up, we kill the following simple lemma.
\begin{lemma}\label{lemma: eliminate equivalent local}
    Let $V$ be an $n$-dimensional vector space, $0\leq s\leq r\leq n$ nonnegative integers. Let $\omega_1,\cdots,\omega_r,\omega$ be positive (1,1)-forms on $V$, $\nu=\omega_1\cdots\omega_s,\ \mu=\omega_1\cdots\omega_r$. Suppose that $\omega\leq N\omega_i,\omega_i\leq N\omega$ for $i=s+1,\cdots,r$. Then
    \begin{align}
        |\alpha|_{(\nu,\omega)}\overset{N}{\approx}|\alpha|_{(\mu,\omega)}
    \end{align}
    for all $\alpha\in H^{p,q}(X)$ with $p+q\leq n-r$ or $p+q\geq n+r$.
\end{lemma}
\begin{proof}
    Apply lemma \ref{lemma: eliminate equivalent global} to an abelian variety of maximum Picard rand and use remark \ref{remark: complex torus = linear}.
\end{proof}

\subsection{Proof of proposition \ref{proposition: local 1}}
\begin{proposition}\label{proposition: local geometric 1}
    Let $X$ be a smooth projective variety of dimension $n$. $h_1,\cdots,h_r,h\in\Amp(X)$. Let $v=h_1\cdots h_r$. Then for all $a\in H^{p,q}(X)$ with $p+q\leq n-r$, we have
    \begin{align}\label{eq41}
        \|va\|_h\leq C\|a\|_{(v,h)}\sqrt{\int_Xvh^{n-r}}
    \end{align}
    for some $C>0$ depending on $X$ and $h$.
\end{proposition}
\begin{proof}
    By continuity it suffices to prove the proposition assuming $h_1,\cdots,h_r\in\Amp(X)\cap N^1(X)_\Q$. Replacing $h_1,\cdots, h_r$ by $Mh_1,\cdots,Mh_r$ for some $M$, we may further assume that $h_1,\cdots, h_r$ are very ample (one can check that replacing $h_i$ by $Mh_i$, both sides of the inequality (\ref{eq41}) are scaled by $M^r$). Let $Y_1,\cdots, Y_r$ be transversally intersecting smooth divisors representing $h_1,\cdots,h_r$. Let $Y=Y_1\cap\cdots\cap Y_r$. Then $\|a\|_{(v,h)}=\|a|_Y\|_h$ by lemma \ref{lemma: restriction identities}. So (\ref{eq41}) is equivalent to
    \begin{align}\label{eq411}
        \|va\|_h\leq C\|a|_Y\|_h\sqrt{\int_Xvh^{n-r}}.
    \end{align}
    To prove this inequality we use the dual point of view. Let $p'=n-p-r,\ q'=n-q-r$, then
    \begin{align*}
        \|va\|_h=\sup_{b\in H^{p',q'}(X),\|b\|_h=1}\int_Xvab
    \end{align*}
    \begin{align*}
        \|a|_Y\|_h=\sup_{b\in H^{p',q'}(Y),\|b\|_h=1}\int_Yab.
    \end{align*}
    Note that
    \[\int_Xvab=\int_Yab\]
    for all $a\in H^{p,q}(X),b\in H^{p',q'}(X)$. So to prove (\ref{eq411}) it suffices to prove
    \begin{align}
        \|b|_Y\|_h\leq C\|b\|_h\sqrt{\int_Xvh^{n-r}}
    \end{align}
    for all $b\in H^{p',q'}(X)$. This follows from lemma \ref{lemma: frequently used}, since
    \[(n-r)!\mathrm{Vol}_\omega(Y)=\int_Y\omega^{n-r}=\int_Xvh^{n-r}.\]
\end{proof}
Now we can prove proposition \ref{proposition: local 1}.
\begin{proof}(of proposition \ref{proposition: local 1})
    Let $e_1,\cdots,e_n$ be an orthonormal basis of $V$ with respect to $\omega$. Let $X=V/\Z e_1+\cdots+\Z e_n$. Then apply proposition \ref{proposition: local geometric 1}.
\end{proof}

\subsection{Proof of proposition \ref{proposition: local 2}}
We first prove a `triangle inequality', which is interesting on its own right.
\begin{lemma}\label{lemma: triangle inequality}
    Let $X$ be a smooth projective variety, $h,w\in\Amp(X)$. Assume that $h$ is very ample and is represented by a simple normal crossing divisor $Y_1+Y_2$ with two components, then for all $a\in H^p(X)$ with $p\geq n-1$, we have
    \begin{align*}
        \|ha\|_{(h,w)}^2\leq\|a|_{Y_1}\|_w^2+\|a|_{Y_2}\|_w^2.
    \end{align*}
\end{lemma}
\begin{remark}
    In the range $p\leq n-1$ the inequality is $\|a\|_{(h,w)}^2\leq\|a|_{Y_1}\|_w^2+\|a|_{Y_2}\|_w^2$. If $Y_1,Y_2$ are very ample this can be written as $\|a\|_{(h_1+h_2,w)}^2\leq\|a\|_{(h_1,w)}^2+\|a\|_{(h_2,w)}^2$, which then holds for all $h_1,h_2\in\Amp(X)$ by standard continuity argument. This is why we call it a `triangle inequality'. It would be interesting to see whether this inequality could be generalized to K\"ahler classes.
\end{remark}
To prove lemma \ref{lemma: triangle inequality} we need two lemmas.
\begin{lemma}\label{lemma: log cut-off}
    Let $(X,\omega)$ be a compact K\"ahler manifold, $Y$ a smooth divisor on $X$. Let $\alpha\in\A^p(X)$ be a smooth $d$-closed form. Then for all $\epsilon>0$, there exists $\beta\in\A^{p-1}(X)$ that vanishes in a neighborhood of $Y$ such that
    \begin{align*}
        \|\alpha+d\beta\|_\omega\leq\|[\alpha]\|_\omega+\epsilon
    \end{align*}
\end{lemma}
\begin{proof}
    The proof uses the `log cut-off' trick. Let $\alpha_H$ be the harmonic representative of $[\alpha]$. Then there exits $\gamma\in\A^{p-1}(X)$ such that $\alpha_H=\alpha+d\gamma$. Define $\beta=\rho\gamma$, where $\rho\in C^{\infty}(X)$ is defined as follow. Fix $r>0$ small so that $Y$ has a tubular $r$-neighborhood in $X$. Define
    \begin{align*}
        \tilde{\rho}(x)=\begin{cases}
            0&d(x,Y)<\delta^2r\\
            \frac{\log\left(d(x,Y)/(\delta^2r)\right)}{|\log\delta|}&\delta^2r\leq d(x,Y)<\delta r\\
            1&d(x,Y)\geq\delta r
        \end{cases}
    \end{align*}
    where $0<\delta\ll 1$. We perturb $\tilde{\rho}$ a little to get $\rho$ so that $\rho$ is smooth. Then
    \[\|d\rho\|_\omega^2\leq C\left|\log\delta\right|^{-1}\]
    \[\|\rho-1\|_\omega^2\leq C\delta^2\]
    For some $C>0$ independent of $\delta$. Then
    \[\|d[(\rho-1)\gamma]\|_\omega=\|(d\rho)\gamma+(\rho-1)d\gamma\|_\omega\leq C'|\log\delta|^{-\frac{1}{2}}.\]
    Take $\delta$ small enough so that $C'|\log\delta|^{-\frac{1}{2}}<\epsilon$. Then
    \[\|\alpha+d\beta\|_\omega=\|\alpha_H+d[(\rho-1)\gamma]\|_\omega\leq\|\alpha_H\|_\omega+\epsilon\]
    as desired.
\end{proof}
\begin{lemma}\label{lemma: singular divisor}
    Let $X$ be a smooth projective variety, $h,w\in\Amp(X)$. Assume that $h$ is very ample and is represented by a simple normal crossing divisor $Y$. Let $\omega$ be a K\"ahler form representing $w$. Then for all $a\in H^p(X)$ with $p\geq n-1$,
    \begin{align}\label{eq48}
        \|ha\|_{(h,w)}^2\leq \inf_{[\alpha]=a}\int_Y|\alpha|_\omega^2\frac{\omega^{n-1}}{(n-1)!}
    \end{align}
    where $\alpha$ ranges over all $d$-closed smooth $p$-forms on $X$ representing $a$.
\end{lemma}
(The assumption that $Y$ is a simple normal crossing divisor is redundant.)
\begin{proof}
    Take a pencil $Y_\lambda$ such that $Y_0=Y$ and $Y_\lambda$ is smooth for generic $\lambda\in\P^1$. By lemma \ref{lemma: restriction identities}, $\|ha\|_{(h,w)}=\|a|_{Y_\lambda}\|_w$ for all $Y_\lambda$ smooth. For all $\alpha\in\A^p(X)$ representing $a$, we have
    \begin{align*}
        \lim_{\lambda\rightarrow0}\int_{Y_\lambda}|\alpha|_\omega^2\frac{\omega^{n-1}}{(n-1)!}=\int_{Y_0}|\alpha|_\omega^2\frac{\omega^{n-1}}{(n-1)!}.
    \end{align*}
    (For a proof of this obvious equation, see the appendix.) Now for all $Y_\lambda$ smooth, we have
    \[\|a|_{Y_\lambda}\|_w^2\leq\int_{Y_\lambda}|\alpha|_\omega^2\frac{\omega^{n-1}}{(n-1)!}.\]
    So we get
    \[\|ha\|_{(h,w)}^2\leq\int_{Y_0}|\alpha|_\omega^2\frac{\omega^{n-1}}{(n-1)!}\]
    for all $\alpha\in\A^p(X)$ representing $a$. This proves (\ref{eq48}).
\end{proof}
Now we can prove lemma \ref{lemma: triangle inequality}.
\begin{proof}
    Let $\omega$ be a K\"ahler form on $X$ representing $w$. By lemma \ref{lemma: singular divisor},
    \begin{align*}
        \|ha\|_{(h,w)}^2\leq \inf_{[\alpha]=a}\int_{Y_1+Y_2}|\alpha|_\omega^2
    \end{align*}
    where $\alpha$ runs over all smooth $d$-closed $p$-forms on $X$ representing $a$. So it suffices to prove that
    \begin{align}\label{eq42}
        \inf_{[\alpha]=a}\int_{Y_1+Y_2}|\alpha|_\omega^2\leq\|a|_{Y_1}\|_w^2+\|a|_{Y_2}\|_w^2.
    \end{align}
    Take an arbitrary $\alpha\in\A^p(X)$ representing $a$. By lemma \ref{lemma: log cut-off}, for every $\epsilon>0$, there exist $\beta_i\in\A^{p-1}(Y_i)\ (i=1,2)$, both vanishing in a neighborhood of $Y_1\cap Y_2$ so that $\|\alpha|_{Y_i}+d\beta_i\|_\omega\leq\|a|_{Y_i}\|_\omega+\epsilon$. It is easy to find $\tilde{\beta_i}\in\A^{p-1}(X)\ (i=1,2)$ so that $\tilde{\beta_i}|_{Y_i}=\beta_i$ and $\tilde{\beta_1}$ vanishes in a neighborhood of $Y_2$, $\tilde{\beta_2}$ vanishes in a neighborhood of $Y_1$. Define $\alpha'=\alpha+d\tilde{\beta_1}+d\tilde{\beta_2}$. Then $\|\alpha'|_{Y_i}\|_\omega\leq\|a|_{Y_i}\|_\omega+\epsilon\ (i=1,2)$. Since $\epsilon>0$ is arbitrary, this proves (\ref{eq42}).
\end{proof}
There are two immediate corollaries.
\begin{corollary}\label{corollary: triangle inequality 1}
    Let $X$ be a smooth projective variety, $h_1,\cdots,h_r,h\in\Amp(X)$. Assume that $h_1,\cdots,h_r$ are all very ample. Let $v=h_2\cdots h_r$. Let $Y^{(1)}+Y^{(2)}$ be a simple normal crossing divisor with two components representing $h_1$. Then for all $a\in H^{p,q}(X)$ with $p+q\geq n-r$,
    \begin{align*}
        \|vh_1a\|_{(vh_1,h)}^2\leq\|va|_{Y^{(1)}}\|_{(v,h)}^2+\|va|_{Y^{(2)}}\|_{(v,h)}^2
    \end{align*}
\end{corollary}
\begin{proof}
    Take smooth divisors $Y_2,\cdots,Y_r$ representing $h_2,\cdots,h_r$ so that $Y^{(1)}+Y^{(2)}+Y_2+\cdots+Y_r$ is a simple normal crossing divisor. Let $Z=Y_2\cap\cdots\cap Y_r$, $Z^{(1)}=Y^{(1)}\cap Z$, $Z^{(2)}=Y^{(2)}\cap Z$. Then $\|vh_1a\|_{(vh_1,h)}=\|h_1a|_Z\|_{(h_1,h)},\ \|va|_{Y^{(1)}}\|_{(v,h)}=\|a|_{Z^{(1)}}\|_h,\ \|va|_{Y^{(2)}}\|_{(v,h)}=\|a|_{Z^{(2)}}\|_h$ by lemma \ref{lemma: restriction identities}. So corollary \ref{corollary: triangle inequality 1} follows from lemma \ref{lemma: triangle inequality} applied on $Z$.
\end{proof}
\begin{corollary}\label{corollary: triangle inequality 2}
    Let $X$ be a smooth projective variety, $h_1,\cdots,h_r,h\in\Amp(X)$. Assume that $h_i$ is ample for every $1\leq i\leq r$ and is represented by $Y_i^{(1)}+Y_i^{(2)}$ such that $Y_1^{(1)}+Y_1^{(2)}+\cdots+Y_r^{(1)}+Y_r^{(2)}$ is a simple normal crossing divisor. Let $v=h_1\cdots h_r$. For every $I=(i_1,\cdots,i_r)\in\{1,2\}^r$, define $Y_I=Y^{(i_1)}\cap\cdots\cap Y^{(i_r)}$. Then for all $a\in H^{p,q}(X)$ with $p+q\geq n-r$,
    \begin{align*}
        \|va\|_{(v,h)}^2\leq\sum_{I\in\{1,2\}^r}\|a|_{Y_I}\|_h^2.
    \end{align*}
\end{corollary}
\begin{proof}
    By corollary \ref{corollary: triangle inequality 1} and induction.
\end{proof}
We are about to succeed. The following proposition is used only to deduce corollary \ref{corollary: perturb nef classes 1} below, where we are going to apply proposition \ref{proposition: perturb nef classes} with $\tilde{h}=\epsilon h$. But for the proof we want to perturb $\epsilon h$ to a rational class and rescale, while leaving $h$ fixed. So we introduce an independent variable $\tilde{h}$ and end up with a bit messy notation.
\begin{proposition}\label{proposition: perturb nef classes}
    Let $X,X'$ be smooth projective varieties, $Y=X\times X'$, $n=\dim Y$ and write $\pi,\pi'$ for the two projections $Y\rightarrow X,\ Y\rightarrow X'$. Let $h_0,h_1,\cdots,h_r\in\Amp(X),\ h'\in\Amp(X'),\ \tilde{h}\in\Amp(Y)$ and suppose that $Nh_0-h_i,N_1h_i-h_0\in\Amp(X)\ (\forall1\leq i\leq r)$, where $N_1\geq N\geq 2$. Let $h=\pi'^*h'+\pi^*h_0\in\Amp(Y),\ v=\pi^*h_1\cdots\pi^*h_r,\ \tilde{v}=(\pi^*h_1+\tilde{h})\cdots(\pi^*h_r+\tilde{h})$. Suppose that $h-\tilde{h}\in\Amp(Y)$. Then for all $a\in H^{p,q}(Y)$ with $p+q\geq n-r$, we have
    \begin{align}\label{eq47}
        \|\tilde{v}a\|_{(\tilde{v},h)}^2\leq N_1^C\|va\|_h^2+C(Y,h)N^n\|a\|_h^2\int_Y\tilde{h}h^{n-1}
    \end{align}
    where $C$ is a constant depending only on $n$, $C(Y,h)$ is a constant depending only on $(Y,h)$.
\end{proposition}
\begin{proof}
    By continuity we may assume that $h_1,\cdots,h_r\in\Amp(X)\cap N^1(X)_\Q,\ \tilde{h}\in\Amp(Y)\cap N^1(Y)_\Q$. Let $M$ be a positive integer so that $Mh_1,\cdots,Mh_r$ are all very ample on $X$, $M\tilde{h}$ is very ample on $Y$. Take transversally intersecting smooth divisors $Y_1,\cdots,Y_r$ on $X$ representing $Mh_1,\cdots, Mh_r$. Let $Y_i^{(2)}=\pi^{-1}(Y_i)$. Take smooth divisors $Y_1^{(1)},\cdots, Y_r^{(1)}$ on $Y$, all representing $M\tilde{h}$, such that $Y_1^{(1)}+Y_1^{(2)}+\cdots+Y_r^{(1)}+Y_r^{(2)}$ is a simple normal crossing divisor. By corollary \ref{corollary: triangle inequality 2} we get
    \begin{align}\label{eq46}
        \|M^r\tilde{v}a\|^2_{(M^r\tilde{v},h)}\leq\sum_{I\in\{1,2\}^r}\|a|_{Y_I}\|_h^2.
    \end{align}
    By lemma \ref{lemma: frequently used},
    \begin{align}\label{eq43}
        \|a|_{Y_I}\|_h^2\leq C'(Y,h)\|a\|_h^2\vol_h(Y_I).
    \end{align}
    If at least one component of $I$ is equal to 1, we have
    \begin{align}\label{eq44}
        \vol_h(Y_I)\leq M^rN^r\int_Y\tilde{h}h^{n-1}.
    \end{align}
    If $I=(2,\cdots,2)$ we have the estimate
    \begin{align}\label{eq45}
        \|a|_{\pi^{-1}(Y_1\cap\cdots\cap Y_r)}\|_h^2\overset{N_1}{\approx}M^r\|va\|_h^2
    \end{align}
    Indeed, let $e_1,\cdots,e_l$ be an orthonormal basis of $H^*(X')$ (equiped with the metric induced by $h'$). (Each $e_j$ has pure degree.) By the K\"unneth theorem we can write $a$ uniquely as $\sum_jb_je_j$ with $b_j\in H^*(X)$. Write $n_j$ for the degree of $b_j$ and let $m=\dim X$. Then
    \begin{align*}
        \begin{split}
            &\|a|_{\pi^{-1}(Y_1\cap\cdots\cap Y_r)}\|_h^2\\
            =&\sum_j\|b_j|_{Y_1\cap\cdots\cap Y_r}\|_{h_0}^2\\
            =&\sum_jM^{n_j-(m-r)}\|b_j|_{Y_1\cap\cdots\cap Y_r}\|_{Mh_0}^2\\
            \overset{N_1}{\approx}&\sum_jM^{n_j-(m-r)}\|M^rh_1\cdots h_rb_j\|_{Mh_0}^2\\=&\sum_jM^{n_j-(m-r)}M^{2r}M^{m-(n_j+2r)}\|h_1\cdots h_rb_j\|_{h_0}^2\\=&\sum_jM^r\|h_1\cdots h_rb_j\|_{h_0}^2\\
            =&M^r\|va\|_h^2
        \end{split}
    \end{align*}
    where we have used lemma \ref{lemma: restriction quasi-isometry 2} and the rescaling laws repeatedly. This proves (\ref{eq45}).

    Combining (\ref{eq43})(\ref{eq44})(\ref{eq45}) we see that
    \begin{align*}
        \sum_{I\in\{1,2\}^r}\|a|_{Y_I}\|_h^2\leq N_1^CM^r\|va\|_h^2+M^rC(Y,h)N^n\|a\|_h^2\int_Y\tilde{h}h^{n-1}.
    \end{align*}
    Plug this into (\ref{eq46}) we get the desired inequality (\ref{eq47}).
\end{proof}
\begin{corollary}\label{corollary: perturb nef classes 1}
    Let $X,X'$ be smooth projective varieties, $Y=X\times X'$, $n=\dim Y$ and write $\pi,\pi'$ for the two projections $Y\rightarrow X,\ Y\rightarrow X'$. Let $h_0,h_1,\cdots,h_r\in\Amp(X),\ h'\in\Amp(X')$ and suppose that $Nh_0-h_i,N_1h_i-h_0\in\Amp(X)\ (\forall1\leq i\leq r)$, where $N_1\geq N\geq 2$. Let $h=\pi'^*h'+\pi^*h_0\in\Amp(Y),\ v=\pi^*h_1\cdots\pi^*h_r$. For every $0<\epsilon<1$ define $v_\epsilon=(\pi^*h_1+\epsilon h)\cdots(\pi^*h_r+\epsilon h)$. Then for all $a\in H^{p,q}(Y)$ with $p+q\geq n-r$, we have
    \begin{align}
        \|v_\epsilon a\|_{(v_\epsilon,h)}^2\leq N_1^C\|va\|_h^2+\epsilon C(Y,h)N^n\|a\|_h^2
    \end{align}
    where $C$ is a constant depending only on $n$, $C(Y,h)$ is a constant depending only on $(Y,h)$.
\end{corollary}
\begin{proof}
    Apply proposition \ref{proposition: perturb nef classes} with $\tilde{h}=\epsilon h$.
\end{proof}
Now we add the assumption $va=0$ and drop the assumption $N_1h_i-h_0\in\Amp(X)$.
\begin{corollary}\label{corollary: perturb nef classes 2}
    Let $X,X'$ be smooth projective varieties, $Y=X\times X'$, $n=\dim Y$ and write $\pi,\pi'$ for the two projections $Y\rightarrow X,\ Y\rightarrow X'$. Let $h_0,h_1,\cdots,h_r\in\Amp(X),\ h'\in\Amp(X')$ and suppose that $Nh_0-h_i\in\Amp(X)\ (\forall1\leq i\leq r)$. Let $h=\pi'^*h'+\pi^*h_0\in\Amp(Y),\ v=\pi^*h_1\cdots\pi^*h_r$. For every $0<\epsilon<1$ define $v_\epsilon=(\pi^*h_1+\epsilon h)\cdots(\pi^*h_r+\epsilon h)$. Let $a\in H^{p,q}(Y)$ with $p+q\geq n-r$. Suppose that $va=0$, then
    \begin{align}
        \|v_\epsilon a\|_{(v_\epsilon,h)}^2\leq\epsilon C(Y,h)N^n\|a\|_h^2
    \end{align}
    where $C(Y,h)$ is a constant depending only on $(Y,h)$.
\end{corollary}
\begin{proof}
    Apply corollary \ref{corollary: perturb nef classes 1} with $N_1$ sufficiently large.
\end{proof}
Now we can prove proposition \ref{proposition: local 2}.
\begin{proof}(of proposition \ref{proposition: local 2})
    Clearly we may assume that $\pi$ is surjective, and we can write $V=W\oplus W'$ with $W, W'$ orthogonal to each other with respect to $\omega$. Let $e_1,\cdots,e_n$ be an orthonormal basis of $V$ with respect to $\omega$ such that $W=\C e_1+\cdots+\C e_m,\ W'=\C e_{m+1}+\cdots+\C e_n$. Let $X=W/\Z e_1+\cdots+\Z e_m,\ X'=W'/\Z e_{m+1}+\cdots+\Z e_n$. Now apply corollary \ref{corollary: perturb nef classes 2}.
\end{proof}

\section{Proof of theorem \ref{theorem 2}}

Notations as in theorem \ref{theorem 2}. Let $\omega_1,\cdots,\omega_s$ be K\"ahler forms on $X$ representing $w_1,\cdots,w_r$ and let $\omega_{s+1},\cdots,\omega_r$ be K\"ahler forms on $Y$ representing $w_{s+1},\cdots,w_{r}$. Let $\nu=\pi^*\omega_1\cdots\pi^*\omega_s,\ \mu=\nu\cdot\omega_{s+1}\cdots\omega_r$.

Suppose $a\in H^{p,q}(Y)$ satisfies $au=0$. We want to show that $av=0$. Let $b=a\cdot w_{s+1}\cdots w_r$. Then $bv=0$.

\textbf{Claim:} there exits $\beta\in\A^{p+r-s,q+r-s}(Y)$ representing $b$ such that $\beta\nu=0$.
\begin{proof}(of claim)
By assumption $\pi^*: H^*(X)\rightarrow H^*(Y)$ makes $H^*(Y)$ into a free $H^*(X)$-module, we take $[\xi_1],\cdots,[\xi_l]\in H^*(Y,\C)$ to be a basis. Then we can write $b$ uniquely as $b=\pi^*b_1\cdot [\xi_1]+\cdots+\pi^*b_l\cdot [\xi_l]$. $b\cdot\pi^*w_1\cdots\pi^*w_s=0$ implies that $b_i\cdot w_1\cdots w_s=0$ for each $i=1,\cdots,l$. By proposition \ref{corollary: pointwise zero}, there exists differential forms $\beta_i$ representing $b_i$ such that $\beta_i\omega_1\cdots\omega_s=0$ for each $i$. Take $\beta=\beta_1\xi_1+\cdots+\beta_l\xi_l$, the claim is proved.
\end{proof}
Let $\omega$ be an arbitrary K\"ahler form on $Y$. Fix $N$ large enough so that $\pi^*\omega_i\leq N\omega\ (i=1,\cdots,s),\ \omega_i\leq N\omega,\omega\leq N\omega_i\ (i=s+1,\cdots,r)$. For $\epsilon>0$, define $\nu_{\epsilon}=(\pi^*\omega_1+\epsilon\omega)\cdots(\pi^*\omega_s+\epsilon\omega),\ \mu_\epsilon=\nu_{\epsilon}\cdot\omega_{s+1}\cdots\omega_r$. Then by proposition \ref{proposition: local 2}, we have the pointwise inequality
\begin{align}\label{eq23}
    |\nu_{\epsilon}\beta|_{(\nu_{\epsilon},\omega)}^2\overset{N}{\lesssim}\epsilon|\beta|_\omega^2.
\end{align}
By lemma \ref{lemma: eliminate equivalent local}, we have $|\nu_{\epsilon}\beta|_{(\nu_{\epsilon},\omega)}^2\overset{N}{\approx}|\nu_{\epsilon}\beta|_{(\mu_\epsilon,\omega)}^2$. So (\ref{eq23}) implies
\begin{align*}
    |\nu_{\epsilon}\beta|_{(\mu_\epsilon,\omega)}^2\overset{N}{\lesssim}\epsilon|\beta|_\omega^2.
\end{align*}
Integrating over $Y$, we get
\begin{align*}
    \int_Y|\nu_{\epsilon}\beta|_{(\mu_\epsilon,\omega)}^2\frac{\omega^n}{n!}\overset{N}{\lesssim}\epsilon\|\beta\|_\omega^2.
\end{align*}
Now by theorem \ref{theorem: harmonic space}, equation (\ref{global norm = local norm}), the above inequality implies that
\begin{align}\label{eq24}
    \|[\nu_{\epsilon}\beta]\|_{([\mu_\epsilon],[\omega])}^2\overset{N}{\lesssim}\epsilon\|\beta\|_\omega^2.
\end{align}
Note that $[\nu_{\epsilon}\beta]=[\mu_\epsilon]\cdot a$, and it is immediate from definition that $\|[\mu_{\epsilon}]\cdot a\|_{([\mu_\epsilon],[\omega])}=\|a\|_{([\mu_\epsilon],[\omega])}$. Plug this into (\ref{eq24}) we get
\begin{align}\label{eq25}
    \|a\|_{([\mu_\epsilon],[\omega])}^2\overset{N}{\lesssim}\epsilon\|\beta\|_\omega^2.
\end{align}
Take $\alpha_\epsilon\in\H^{p,q}_{(\mu_\epsilon,\omega)}(Y)$ representing $a$. By theorem \ref{theorem: harmonic space}, we have
\begin{align}\label{eq26}
    \|a\|_{([\mu_\epsilon],[\omega])}^2=\int_Y|\alpha_\epsilon|^2_{(\mu_\epsilon,\omega)}\frac{\omega^n}{n!}.
\end{align}
By lemma \ref{lemma: eliminate equivalent local}, we have
\begin{align}\label{eq27}
    |\alpha_\epsilon|_{(\nu_{\epsilon},\omega)}\overset{N}{\approx}|\alpha_\epsilon|_{(\mu_\epsilon,\omega)}
\end{align}
By proposition \ref{proposition: local 1}, we have
\begin{align}\label{eq28}
    |\nu_{\epsilon}\alpha_\epsilon|_\omega\lesssim N^{r/2}|\alpha_\epsilon|_{(\nu_{\epsilon},\omega)}
\end{align}
Combining (\ref{eq26})(\ref{eq27})(\ref{eq28}) we get
\begin{align*}
    \int_Y|\nu_{\epsilon}\alpha_\epsilon|_\omega^2\frac{\omega^n}{n!}\overset{N}{\lesssim}\|a\|_{([\mu_\epsilon],[\omega])}^2.
\end{align*}
LHS of the above inequality is greater than or equal to $\|[\nu_{\epsilon}\alpha_\epsilon]\|_\omega^2$ (classical Hodge theory, or the baby case of theorem \ref{theorem: harmonic space}). Recall that $[\alpha_\epsilon]=a$, so $[\nu_{\epsilon}\alpha_\epsilon]=[\nu_{\epsilon}]\cdot a$. So the above inequality implies
\begin{align*}
    \|[\nu_{\epsilon}]\cdot a\|_\omega^2\overset{N}{\lesssim}\|a\|_{([\mu_\epsilon],[\omega])}^2.
\end{align*}
Plug this into (\ref{eq25}) we get
\begin{align*}
    \|[\nu_{\epsilon}]\cdot a\|_\omega^2\overset{N}{\lesssim}\epsilon\|\beta\|_\omega^2.
\end{align*}
The map $[0,1)\rightarrow H^{p+s,q+s}(Y),\ \epsilon\mapsto [\nu_{\epsilon}]=(\pi^*w_1+\epsilon[\omega])\cdots(\pi^*w_s+\epsilon[\omega])$ is continuous, so we can let $\epsilon\rightarrow 0$ in the above inequality and get $\|va\|_\omega=0$. So we must have $va=0$, finishing the proof.

Theorem \ref{theorem 4} is proven in exaclty the same way.

\section{Proof of theorem \ref{theorem 3}}
We allow $E$ to be $\R$-twisted. By dimension reason, in order to prove that $-\wedge c_e(E)w_1\cdots w_k:H^{p,q}(X)\rightarrow H^{n-q,n-p}(X)$ is an isomorphism, it suffices to prove that it is injective. Let $a\in H^{p,q}(X)$ such that $a\cdot c_e(E)w_1\cdots w_k=0$, we want to show that $a=0$. Let $P=\P(E)$, $\pi:P\rightarrow X$ be the projection. Let $w=c_1(\O_{\P(E)}(1))\in\Amp(P)$. Then
\begin{align*}
    w^e-\pi^*c_1(E)w^{e-1}+\cdots+(-1)^e\pi^*c_e(E)=0.
\end{align*}
We get
\begin{align*}
    (\pi^*a\cdot\gamma)\cdot w\cdot\pi^*w_1\cdots\pi^*w_k=\pi^*\left(a\cdot c_e(E)w_1\cdots w_k\right)=0
\end{align*}
where
\[\gamma=(-1)^{e-1}w^{e-1}+(-1)^{e-2}\pi^*c_1(E)w^{e-2}\cdots+\pi^*c_{e-1}(E).\]
By theorem \ref{theorem 2}, this implies
\[(\pi^*a\cdot\gamma)\cdot\pi^*w_1\cdots\pi^*w_k=0.\]
Recall that $H^*(P)$ is a free $H^*(X)$ module with basis $1,w,\cdots,w^{e-1}$. Comparing coefficients of $w^{e-1}$ we get
\[a\cdot w_1\cdots w_k=0.\]
But $-\wedge w_1\cdots w_k: H^{p,q}(X)\rightarrow H^{p+k,q+k}(X)$ is injective (say by theorem \ref{theorem: harmonic space}). So $a=0$ as desired. This finishes the proof of theorem \ref{theorem 3}. There is also a Hodge-Riemann relation (see the more general version below). The proof of the Hodge-Riemann relation is the same as in theorem \ref{theorem 1} and we do not repeat it again.

Theorem \ref{theorem 3} is more natural in the more general setting of K\"ahler vector bundles.
\begin{definition}
    Let $X$ be a compact K\"ahler manifold. A \textit{(1,1)-twisted vector bundle} $E\left<w\right>$ on $X$ is an ordered pair consisting of a vector bundle $E$ on $X$ and a class $w\in H^{1,1}(X,\R)$. It is called a K\"ahler vector bundle if $c_1(\O_{\P(E)}(1))+\pi^*w$ is a K\"ahler class on $\P(E)$. 
\end{definition}
The Chern classes of a (1,1)-twisted vector bundle is defined in the same way as an $\R$-twisted vector bundle. The proof of theorem \ref{theorem 3} generalizes without change to prove the following theorem.
\begin{theorem}
    Let $X$ be a compact K\"ahler manifold of dimension $n$, $E$ a K\"ahler vector bundle on $X$ of rank $e$, $w_1,\cdots,w_k\in H^{1,1}(X,\R)$ K\"ahler classes. Let $p,q$ be nonnegative integers satisfying $p+q+e+k=n$. Then
    \begin{itemize}
        \item[(1)] (Hard Lefschetz)
        \[-\wedge c_e(E)w_1\cdots w_k:H^{p,q}(X)\rightarrow H^{n-q,n-p}(X)\]
        is an isomorphism.
        \item[(2)] (Hodge-Riemann relation) Let $w\in H^{1,1}(X,\R)$ be an arbitrary K\"ahler class. Then The Hermitian form on $H^{p,q}(X)$
        \[\left<a,b\right>=(-1)^q\sqrt{-1}^{(p+q)^2}\int_Xa\bar{b}\cdot c_e(E)w_1\cdots w_k\]
        is positive definite on the subspace
        \[P^{p,q}(X)=\ker\left(-\wedge c_e(E)w_1\cdots w_kw:H^{p,q}(X)\rightarrow H^{n-q+1,n-p+1}(X)\right).\]
    \end{itemize}
\end{theorem}
The natural generalization of theorem \ref{theorem 1} to K\"ahler vector bundles, which contains the above theorem as a special case, is the following
\begin{conjecture}
    Let $X$ be a compact K\"ahler manifold of dimension $n$, $E_1,\cdots,E_k$ K\"ahler vector bundles on $X$. Let $p,q$ be nonnegative integers satisfying $p+q+e_1+\cdots+e_k=n$. Then $c_{e_1}(E_1)\cdots c_{e_k}(E_k)$ satisfies the hard Lefschetz property and the Hodge-Riemann relation.
\end{conjecture}

\section{Appendix}
\begin{lemma}
    Let $X$ be a smooth projective variety and let $Y_\lambda$ be a pencil whose generic element is smooth. Then for any $\alpha\in\A^p(X)$, any K\"ahler metric $\omega$ on $X$, we have
    \begin{align*}
        \lim_{\lambda\rightarrow 0}\int_{Y_\lambda}|\alpha|_\omega^2\frac{\omega^{n-1}}{(n-1)!}=\int_{Y_0}|\alpha|_\omega^2\frac{\omega^{n-1}}{(n-1)!}.
    \end{align*}
\end{lemma}
\begin{proof}
    The convergence away from the singular loci of $(Y_0)_{red}$ is clear. It suffices to prove the following statement. For every $\epsilon>0$, there exists a neighborhood $N(\epsilon)$ of the singular loci of $(Y_0)_{red}$ such that $\vol_\omega(Y_\lambda\cap N(\epsilon))<\epsilon$ for all $|\lambda|<\delta(\epsilon)$. In the case that we need, $Y_0$ is a simple normal crossing divisor, and the above statement can be seen easily from local equations.
\end{proof}

\printbibliography

\end{document}